\documentclass[preprint,3p]{elsarticle}%
\usepackage{amsfonts}
\usepackage{amssymb}
\usepackage{amsmath}
\usepackage{amsthm}
\usepackage{graphicx}
\providecommand{\U}[1]{\protect\rule{.1in}{.1in}}

\newtheorem{theorem}{Theorem}

\newtheorem{definition}[theorem]{Definition}

\newtheorem{lemma}[theorem]{Lemma}

\newtheorem{proposition}[theorem]{Proposition}

\journal{}

\begin{document}

\begin{frontmatter}

\title{Local Linearization - Runge Kutta Methods: a class of A-stable explicit integrators for dynamical systems}

\author[impa,uci]{H. de la Cruz\corref{cor1}}
\ead{hugo@impa.br}
\author[uval,icimaf]{R.J. Biscay}
\ead{rolando.biscay@uv.cl}
\author[icimaf]{J.C. Jimenez}
\ead{jcarlos@icmf.inf.cu}
\author[mgill]{F. Carbonell}
\ead{felix.carbonell@mail.mcgill.com}

\address[impa]{IMPA, Estrada Dona Castorina 110, Rio de Janeiro, Brasil}
\address[uci]{Universidad de Ciencias Inform\'aticas, La Habana, Cuba}
\address[uval]{CIMFAV-DEUV, Facultad de Ciencias, Universidad de Valparaiso, Chile}
\address[icimaf]{Instituto de Cibern\'etica, Matem\'atica y F\'isica, Calle 15 No. 551, La Habana, Cuba}
\address[mgill]{Montreal Neurological Institute, McGill University, Montreal, Canada}

\cortext[cor1]{Corresponding author}
\fntext[]{Supported by CNPq under grant no. 500298/2009-2}

\begin{abstract}
A new approach for the construction of high order A-stable explicit
integrators for ordinary differential equations (ODEs) is theoretically
studied. Basically, the integrators are obtained by splitting, at each time
step, the solution of the original equation in two parts: the solution of a
linear ordinary differential equation plus the solution of an auxiliary ODE.
The first one is solved by a Local Linearization scheme in such a way that
A-stability is ensured, while the second one can be approximated by any
extant scheme, preferably a high order explicit Runge-Kutta scheme. Results
on the convergence and dynamical properties of this new class of schemes are
given, as well as some hints for their efficient numerical implementation.
An specific scheme of this new class is derived in detail, and its
performance is compared with some Matlab codes in the integration of a
variety of ODEs representing different types of dynamics.

\end{abstract}


\end{frontmatter}

\section{Introduction}

It is well known (see, i.e., \cite{Cartwright 1992, Stewart 1992}) that
conventional numerical schemes such as Runge-Kutta, Adams-Bashforth,
predictor-corrector and others produce misleading dynamics in the
integration of Ordinary Differential Equations (ODEs). Typical difficulties
are, for instance, the convergence to spurious steady states, changes in the
basis of attraction, appearance of spurious bifurcations, etc. The essence
of such difficulties is that the dynamics of the numerical schemes (viewed
as discrete dynamical systems) is far richer than that of its continuous
counterparts. Contrary to the common belief, drawbacks of this type may not
be solved by reducing the step-size of the numerical method. Therefore, it
is highly desirable the development of numerical integrators that preserve,
as much as possible, the dynamical properties of the underlaying dynamical
system for all step sizes or relative big ones. In this direction, some
modest advances has been archived by a number of relative recent integrators
of the class of Exponential Methods, which are characterized by the explicit
use of exponentials to obtain an approximate solution. In fact, their
development has been encouraged because their capability of preserving a
number of geometric and dynamical features of the ODEs at the expense of
notably less computational effort than implicit integrators. This have
become feasible due to advances in the computation of matrix exponentials
(see, e.g., \cite{Hochbruck-Lubich97}, \cite{Sidje 1998}, \cite%
{Dieci-Papini00}, \cite{Celledoni-Iserles01}, \cite{Higham04}) and multiple
integrals involving matrix exponentials (see, e.g., \cite{Carbonell-etal05},
\cite{VanLoan78}). Some instances of this type of integrators are the
methods known as exponential fitting \cite{Liniger70}, \cite{Carroll93},
\cite{Voss 1988}, \cite{Cash 1981}, \cite{Iserles 1978}, exponential
integrating factor \cite{Lawson67}, exponential integrators \cite%
{Hochbruck-etal98},\cite{Hochbruck-Ostermann10}, exponential time
differencing \cite{Cox-Matthews02}, \cite{Kassam05}, truncated Magnus
expansion \cite{Iserles99}, \cite{Blanes-etal00}, truncated Fer expansion
\cite{Zanna99} (also named exponential of iterated commutators in \cite%
{Iserles 1984}), exponential Runge-Kutta \cite{Hochbruck-Ostermann04}, \cite%
{Hochbruck-Ostermann05}, some schemes based on versions of the variation of
constants formula (e.g., \cite{Norsett69}, \cite{Jain72}, \cite{Iserles81},
\cite{Pavlov-Rodionova87}, \cite{Friesner-etal89}), local linearization
(see, e.g., \cite{Pope63}, \cite{Ramos 1997a}, \cite{Jimenez02 AMC}, \cite%
{Jimenez05 AMC}, \cite{Carr11}), and high order local linearization methods
\cite{de la Cruz 06}, \cite{de la Cruz 07}, \cite{Jimenez09}, \cite%
{Hochbruck-Ostermann11}.

The present paper deals with the class of high order local linearization
integrators called Local Linearization-Runge Kutta (LLRK) methods, which was
recently introduced in \cite{de la Cruz 06} as a flexible approach for
increasing the order of convergence of the Local Linearization (LL) method
while retaining its desired dynamical properties. Essentially, the LLRK
integrators are obtained by splitting, at each time step, the solution of
the underlying ODE in two parts: the solution $\mathbf{v}$ of a linear ODE
plus the solution $\mathbf{u}$ of an auxiliary ODE. The first one is solved
by an LL scheme in such a way that the A-stability is ensured, while the
second one is integrated by any high order explicit Runge-Kutta (RK) scheme.
Likewise Implicit-Explicit Runge-Kutta (IMEX RK) and conventional splitting
methods (see e.g. \cite{McLachlan02}, \cite{Ascher-etal97}), the splitting
involved in the LLRK approximations is based on the representation of the
underlying vector field as the addition of linear and nonlinear components.
However, there are notable differences among these methods: i) Typically, in
splitting and IMEX methods the vector field decomposition is global instead
of local, and it is not based on a first-order Taylor expansion. ii) In
contrast with IMEX and LLRK approaches, splitting methods construct an
approximate solution by composition of the flows corresponding to the
component vector fields. iii) IMEX RK methods are partitioned (more
specifically, additive) Runge-Kutta methods that compute a solution $\mathbf{%
y}=\mathbf{v+u}$ by solving certain ODE for $\left( \mathbf{v,u}\right) $,
setting different RK coefficients for each block. LLRK methods also solve a
partitioned system for $\left( \mathbf{v,u}\right) $, but a different one.
In this case, one of the blocks is linear and uncoupled, which is solved by
the LL method. After inserting the (continuous time) LL approximation into
the second block, this is treated as a non-autonomous ODE, for which any
extant RK discretization can be used. On the other hand, it is worth noting
that the LLRK methods can also be thought of a flexible approach to
construct new A-stable explicit schemes based on standard explicit RK
integrators. In comparison with the well known Rosenbrock \cite{Bui79}, \cite%
{Shampine97} and Exponential Integrators \cite{Hochbruck-etal98},\cite%
{Hochbruck-Ostermann05} the A-stability of the LLRK schemes is achieved in a
different way. Basically, Rosenbrock and Exponential integrators are
obtained by inserting a stabilization factor ($1/(1-z)$ or $(e^{z}-1)/z$,
respectively) into the explicit RK formulas, whose coefficients must then be
determined to fulfil both A-stability and order conditions. In contrast,
A-stability of an LLRK scheme results from the fact that the component $%
\mathbf{v}$ associated with the linear part of the vector field is computed
through an A-stable LL scheme. Another major difference is that the RK
coefficients involved in the LLRK methods are not constrained by any
stability condition and they just need satisfy the usual order conditions
for RK schemes. Thus, the coefficients in the LLRK methods can be just those
of any standard explicit RK scheme. This makes the LLRK approach greatly
flexible and allows for simple numerical implementations on the basis of
available subroutines for LL and RK methods.

In \cite{de la Cruz 06}, \cite{de la Cruz Ph.D. Thesis} a number of
numerical simulations were carried out in order to illustrate the
performance of the LLRK schemes and to compare them with other numerical
integrators. With special emphasis, the dynamical properties of the LLRK
schemes were considered, as well as, their capability for integrating some
kinds of stiff ODEs. For these equations, LLRK schemes showed stability
similar to that of implicit schemes with the same order of convergence,
while demanding much lower computational cost. The simulations also showed
that the LLRK schemes exhibit a much better behavior near stationary
hyperbolic points and periodic orbits of the continuous systems than others
conventional explicit integrators. However, no theoretical support to such
findings has been published so far.

The main aim of the present paper is to provide a theoretical study of LLRK
integrators.\ Specifically, the following subjects are considered: rate of
convergence, linear stability, preservation of the equilibrium points, and
reproduction of the phase portrait of the underlying dynamical system near
hyperbolic stationary points and periodic orbits. Furthermore, unlike the
majority of the previous papers on exponential integrators, this study is
carried out not only for the discretizations but also for the numerical
schemes that implement them in practice.

The paper is organized as follows. In section 2, the formulations of the LL
and LLRK methods are briefly reviewed. Sections 3 and 4 deal with the
convergence, linear stability and dynamic properties of LLRK
discretizations. Section 5 focuses on the preservation of these properties
by LLRK numerical schemes. In the last section, a new simulation study is
presented in order to compare the performance of an specific order 4 LLRK
scheme and some Matlab codes in a variety of ODEs representing different
types of dynamics.

\section{High Order Local Linear discretizations \label{Section LLA}}

Let $\mathcal{D}\subset\mathbb{R}^{d}$ be an open set. Consider the $d$%
-dimensional differential equation%
\begin{align}
\frac{d\mathbf{x}\left( t\right) }{dt} & =\mathbf{f}\left( t,\mathbf{x}%
\left( t\right) \right) \text{, \ \ }t\in\left[ t_{0},T\right]
\label{ODE-LLA-1} \\
\mathbf{x}(t_{0}) & =\mathbf{x}_{0},  \label{ODE-LLA-2}
\end{align}
where $\mathbf{x}_{0}\in\mathcal{D}$ is a given initial value, and $\mathbf{f%
}:\left[ t_{0},T\right] \times\mathcal{D}\longrightarrow \mathbb{R}^{d}$ is
a differentiable function. Lipschitz and smoothness conditions on the
function $\mathbf{f}$ are assumed in order to ensure a unique solution of
this equation in $\mathcal{D}$.

In what follows, for $h>0$, $(t)_{h}$ will denote a partition $%
t_{0}<t_{1}<...<t_{N}=T$ of the time interval $\left[ t_{0},T\right] $ such
that
\begin{equation*}
\underset{n}{sup}(h_{n})\leq h<1,
\end{equation*}
where $h_{n}=t_{n+1}-t_{n}$ for $n=0,...,N-1$.

\subsection{Local Linear discretization}

Suppose that, for each $t_{n}\in\left( t\right) _{h}$, $\mathbf{y}_{n}\in%
\mathcal{D}$ is a point close to $\mathbf{x}\left( t_{n}\right) $. Consider
the first order Taylor expansion of the function $\mathbf{f}$ around the
point $(t_{n},\mathbf{y}_{n})$:

\begin{equation*}
\mathbf{f}\left( s,\mathbf{u}\right) \approx\mathbf{f(}t_{n},\mathbf{y}_{n})+%
\mathbf{f}_{\mathbf{x}}(t_{n},\mathbf{y}_{n})(\mathbf{u}-\mathbf{y}_{n})+%
\mathbf{f}_{t}(t_{n},\mathbf{y}_{n})(s-t_{n}),\text{ }
\end{equation*}
for $s\in\mathbb{R}$ and $\mathbf{u}\in\mathcal{D}$, where $\mathbf{f}_{%
\mathbf{x}}$, and $\mathbf{f}_{t}$ denote the partial derivatives of $%
\mathbf{f}$ with respect to the variables $\mathbf{x}$ and $t$,
respectively. Adopting this linear approximation of $\mathbf{f}$ at each
time step, the solution of (\ref{ODE-LLA-1})-(\ref{ODE-LLA-2}) can be
locally approximated on each interval $[t_{n},t_{n+1})$ by the solution of
the linear ODE%
\begin{align}
\frac{d\mathbf{y}\left( t\right) }{dt} & =\mathbf{A}_{n}\mathbf{y}(t)+%
\mathbf{a}_{n}\left( t\right) \text{, \ \ }t\in\lbrack t_{n},t_{n+1})
\label{ODE-LLA-8} \\
\mathbf{y}\left( t_{n}\right) & =\mathbf{y}_{n}\text{ , \ \ }
\label{ODE-LLA-8b}
\end{align}
where $\mathbf{A}_{n}=\mathbf{f}_{\mathbf{x}}(t_{n},\mathbf{y}_{n})$ is a
constant matrix, $\mathbf{a}_{n}(t)=\mathbf{f}_{t}\left( t_{n},\mathbf{y}%
_{n}\right) (t-t_{n})+\mathbf{f}\left( t_{n},\mathbf{y}_{n}\right) \mathbf{-A%
}_{n}\mathbf{y}_{n}$ is a linear vector function of $t$. According to the
variation of constants formula, such a solution is given by
\begin{equation}
\mathbf{y}(t)=e^{\mathbf{A}_{n}(t-t_{n})}(\mathbf{y}_{n}+\int%
\limits_{0}^{t-t_{n}}e^{-\mathbf{A}_{n}u}\mathbf{a}_{n}\left( t_{n}+u\right)
du).  \label{ODE-LLA-6}
\end{equation}
Furthermore, by using the identity%
\begin{equation}
\int\limits_{0}^{\Delta}e^{-\mathbf{A}_{n}u}du\text{ }\mathbf{A}_{n}=-(e^{-%
\mathbf{A}_{n}\Delta}-\mathbf{I}),\text{ \ \ \ \ }\Delta\geq0
\label{ODE-LLA-3}
\end{equation}
and simple rules from the integral calculus, the expression (\ref{ODE-LLA-6}%
) can be rewritten as
\begin{equation}
\mathbf{y}(t)=\mathbf{y}_{n}+\mathbf{\phi}(t_{n},\mathbf{y}_{n};t-t_{n}),
\label{ODE-LLA-9}
\end{equation}
where%
\begin{align}
\mathbf{\phi}(t_{n},\mathbf{y}_{n};t-t_{n}) & =\int\limits_{0}^{t-t_{n}}e^{%
\mathbf{A}_{n}(t-t_{n}-u)}(\mathbf{A}_{n}\mathbf{y}_{n}+\mathbf{a}_{n}\left(
t_{n}+u)\right) du  \notag \\
& =\int\limits_{0}^{t-t_{n}}e^{\mathbf{f}_{\mathbf{x}}\left( t_{n},\mathbf{y}%
_{n}\right) (t-t_{n}-u)}(\mathbf{f}\left( t_{n},\mathbf{y}_{n}\right) +%
\mathbf{f}_{t}\left( t_{n},\mathbf{y}_{n}\right) u)du.  \label{ODE-LLA-7}
\end{align}

In this way, by setting $\mathbf{y}_{0}=\mathbf{x}(t_{0})$ and iteratively
evaluating the expression (\ref{ODE-LLA-9}) at $t_{n+1}$ (for $n=0,1,\ldots
,N-1$) a sequence of points $\mathbf{y}_{n+1}$ can be obtained as an
approximation to the solution of the equation (\ref{ODE-LLA-1})-(\ref%
{ODE-LLA-2}). This is formalized in the following definition.

\begin{definition}
\label{definition LLD} (\cite{Jimenez02}, \cite{Jimenez05 AMC}) For a given
time discretization $\left( t\right) _{h}$, the Local Linear discretization
for the ODE (\ref{ODE-LLA-1})-(\ref{ODE-LLA-2}) is defined by the recursive
expression%
\begin{equation}
\mathbf{y}_{n+1}=\mathbf{y}_{n}+\mathbf{\phi}\left( t_{n},\mathbf{y}%
_{n};h_{n}\right) ,  \label{ODE-LLA-4}
\end{equation}
starting with $\mathbf{y}_{0}=\mathbf{x}_{0}$.
\end{definition}

The Local Linear discretization\ (\ref{ODE-LLA-4}) is, by construction,
A-stable. Furthermore, under quite general conditions, it does not have
spurious equilibrium points \cite{Jimenez02 AMC} and preserves the local
stability of the exact solution at hyperbolic equilibrium points and
periodic orbits \cite{Jimenez02 AMC}, \cite{McLachlan09}. On the basis of
the recursion (\ref{ODE-LLA-4}) (also known as Exponentially fitted Euler,
Euler Exponential or piece-wise linearized method) a variety of numerical
schemes for ODEs has been constructed (see a review in \cite{Jimenez05 AMC},
\cite{de la Cruz 07}). These numerical schemes essentially differ with
respect to the numerical algorithm used to compute (\ref{ODE-LLA-7}), and so
in the dynamical properties that they inherit from the LL discretization. A
major limitation of such schemes is their low order of convergence, namely
two.

\subsection{Local Linear - Runge Kutta discretizations}

A modification of the classical LL method can be done in order to improve
its order of convergence while retaining desirable dynamic properties. To do
so, note that the solution of the local linear ODE (\ref{ODE-LLA-8})-(\ref%
{ODE-LLA-8b}) is an approximation to the solution of the local nonlinear ODE
\begin{align*}
\frac{d\mathbf{z}\left( t\right) }{dt}& =\mathbf{f}\left( t,\mathbf{z}\left(
t\right) \right) \text{, \ \ }t\in \lbrack t_{n},t_{n+1}) \\
\mathbf{z}\left( t_{n}\right) & =\mathbf{y}_{n},\text{\ \ }
\end{align*}%
which can be rewritten as
\begin{align*}
\frac{d\mathbf{z}\left( t\right) }{dt}& =\mathbf{A}_{n}\mathbf{z}(t)+\mathbf{%
a}_{n}\left( t\right) +\mathbf{g}(t_{n},\mathbf{y}_{n};t,\mathbf{z}\left(
t\right) )\text{, \ \ }t\in \lbrack t_{n},t_{n+1}) \\
\mathbf{z}\left( t_{n}\right) & =\mathbf{y}_{n},\text{\ \ }
\end{align*}%
where $\mathbf{g}(t_{n},\mathbf{y}_{n};t,\mathbf{z}\left( t\right) )=\mathbf{%
f(}t,\mathbf{z}\left( t\right) )-\mathbf{A}_{n}\mathbf{z}(t)-\mathbf{a}%
_{n}\left( t\right) $, and $\mathbf{A}_{n}$, $\mathbf{a}_{n}(t)$ are defined
as in the previous subsection. From the variation of constants formula, the
solution $\mathbf{z}$ of this equation can be written as
\begin{equation*}
\mathbf{z}\left( t\right) =\mathbf{y}_{LL}\left( t;t_{n},\mathbf{y}%
_{n}\right) +\mathbf{r}\left( t;t_{n},\mathbf{y}_{n}\right) ,
\end{equation*}%
where
\begin{equation}
\mathbf{y}_{LL}(t;t_{n},\mathbf{y}_{n})=e^{\mathbf{A}_{n}(t-t_{n})}(\mathbf{y%
}_{n}+\int\limits_{0}^{t-t_{n}}e^{-\mathbf{A}_{n}u}\mathbf{a}_{n}\left(
t_{n}+u\right) du)  \label{ODE_HLLA-2b}
\end{equation}%
is solution of the linear equation (\ref{ODE-LLA-8})-(\ref{ODE-LLA-8b}) and
\begin{equation}
\mathbf{r}(t;t_{n},\mathbf{y}_{n})=\int\limits_{0}^{t-t_{n}}e^{\mathbf{f}_{%
\mathbf{x}}(t_{n},\mathbf{y}_{n})(t-t_{n}-u)}\mathbf{g}\left( t_{n},\mathbf{y%
}_{n};t_{n}+u,\mathbf{z}\left( t_{n}+u\right) \right) du  \label{ODE-HLLA-2}
\end{equation}%
is the remainder term of the LL approximation $\mathbf{y}_{LL}$ to $\mathbf{z%
}$. Consequently, if $\mathbf{r}_{\kappa }$ is an approximation to $\mathbf{r%
}$ of order $\kappa >2$, then $\mathbf{y}(t)=\mathbf{y}_{LL}\left( t;t_{n},%
\mathbf{y}_{n}\right) +\mathbf{r}_{\kappa }\left( t;t_{n},\mathbf{y}%
_{n}\right) $ should provide a better estimate to $\mathbf{z}(t)$ than the
LL approximation $\mathbf{y}(t)=\mathbf{y}_{LL}\left( t;t_{n},\mathbf{y}%
_{n}\right) $ for all $t\in \lbrack t_{n},t_{n+1})$. This motivates the
definition of the following high order local linear discretization.

\begin{definition}
\label{definition HLLD}\cite{Jimenez09} For a given time discretization $%
\left( t\right) _{h}$, an order $\gamma $ Local Linear discretization for
the ODE (\ref{ODE-LLA-1})-(\ref{ODE-LLA-2}) is defined by the recursive
expression%
\begin{equation}
\mathbf{y}_{n+1}=\mathbf{y}_{LL}\left( t_{n}+h_{n};t_{n},\mathbf{y}%
_{n}\right) +\mathbf{r}_{\kappa }\left( t_{n}+h_{n};t_{n},\mathbf{y}%
_{n}\right) ,  \label{ODE-HLLA-1}
\end{equation}%
starting with $\mathbf{y}_{0}=\mathbf{x}_{0}$, where $\mathbf{r}_{\kappa }$
is an approximation to the remainder term (\ref{ODE-HLLA-2}) such that $%
\left\Vert \mathbf{x}(t_{n})-\mathbf{y}_{n}\right\Vert =O(h^{\gamma })$ with
$\gamma >2$, for all $t_{n}\in \left( t\right) _{h}$.
\end{definition}

Depending on the way in which the remainder term $\mathbf{r}$ is
approximated, two classes of high order LL discretizations have been
proposed.\ In the first one, $\mathbf{g}$ is approximated by a polynomial.
For instance, by means of a truncated Taylor expansion \cite{de la Cruz 07}
or an Hermite interpolation polynomial \cite{Hochbruck-Ostermann11},
resulting in the so called Local Linearization - Taylor schemes and the
Linearized Exponential Adams schemes, respectivelly. The second one is based
on approximating $\mathbf{r}$ by means of a standard integrator that solves
an auxiliary ODE. This is called the Local Linearization-Runge Kutta (LLRK)
methods when a Runge-Kutta integrator is used for this purpose \cite{de la
Cruz 06}. A computational advantage of the latter class is that it does not
require calculation of high order derivatives of the vector field $\mathbf{f}
$.

Specifically, the LLRK methods are derived as follows. By taking derivatives
with respect to $t$ in (\ref{ODE-HLLA-2}), it is obtained that $\mathbf{r}%
\left( t;t_{n},\mathbf{y}_{n}\right) $ satisfies the differential equation%
\begin{align}
\frac{d\mathbf{u}\left( t\right) }{dt}& =\mathbf{q(}t_{n},\mathbf{y}_{n};t%
\mathbf{,\mathbf{u}}\left( t\right) \mathbf{),}\text{ \ \ }t\in \lbrack
t_{n},t_{n+1}),  \label{Diff. Equat for rn} \\
\mathbf{u}\left( t_{n}\right) & =\mathbf{0},
\label{Initial Cond. Diff. Equat for rn}
\end{align}%
with vector field
\begin{equation*}
\mathbf{q(}t_{n},\mathbf{y}_{n};s\mathbf{,\xi )}=\mathbf{\mathbf{f}_{\mathbf{%
x}}}(t_{n},\mathbf{y}_{n})\mathbf{\xi }+\mathbf{g}\left( t_{n},\mathbf{y}%
_{n};s,\mathbf{y}_{n}+\mathbf{\phi }\left( t_{n},\mathbf{y}%
_{n};s-t_{n}\right) +\mathbf{\xi }\right) ,
\end{equation*}%
which can be also written as
\begin{align*}
\mathbf{q(}t_{n},\mathbf{y}_{n};s\mathbf{,\xi )}=\mathbf{f(}& s,\mathbf{y}%
_{n}+\mathbf{\phi }\left( t_{n},\mathbf{y}_{n};s-t_{n}\right) +\mathbf{\xi }%
)-\mathbf{f}_{\mathbf{x}}(t_{n},\mathbf{y}_{n})\mathbf{\phi }\left( t_{n},%
\mathbf{y}_{n};s-t_{n}\right) \\
& -\mathbf{f}_{t}\left( t_{n},\mathbf{y}_{n}\right) (s-t_{n})-\mathbf{f}%
\left( t_{n},\mathbf{y}_{n}\right) ,
\end{align*}%
where $\mathbf{\phi }$ is the vector function (\ref{ODE-LLA-7}) that defines
the LL discretization (\ref{ODE-LLA-4}). Thus, an approximation $\mathbf{r}%
_{\kappa }$ to $\mathbf{r}$ can be obtained by solving the ODE (\ref{Diff.
Equat for rn})-(\ref{Initial Cond. Diff. Equat for rn}) through any
conventional numerical integrator. Namely, if $\mathbf{u}_{n+1}=\mathbf{u}%
_{n}+\mathbf{\Lambda }^{\mathbf{y}_{n}}\left( t_{n},\mathbf{u}%
_{n};h_{n}\right) $ is some one-step numerical scheme for this equation,
then $\mathbf{r}_{\kappa }\left( t_{n}+h_{n};t_{n},\mathbf{y}_{n}\right) =%
\mathbf{\Lambda }^{\mathbf{y}_{n}}\left( t_{n},\mathbf{0};h_{n}\right) $.

In particular, we will focus on the approximation $\mathbf{r}_{\kappa}$
obtained by means of an explicit RK scheme of order $\kappa$. Consider an
s-stage explicit RK scheme with coefficients $\mathbf{c}=\left[ c_{i}\right]
$, \ $\mathbf{A}=\left[ a_{ij}\right] $, \ $\mathbf{b}=\left[ b_{j}\right] $
applied to the equation (\ref{Diff. Equat for rn})-(\ref{Initial Cond. Diff.
Equat for rn}), i.e., the approximation defined by the map
\begin{equation}
\mathbf{\rho}\left( t_{n},\mathbf{y}_{n};h_{n}\right)
=h_{n}\sum_{j=1}^{s}b_{j}\mathbf{k}_{j},  \label{RK}
\end{equation}
where%
\begin{equation*}
\mathbf{k}_{i}=\mathbf{q}(t_{n},\mathbf{y}_{n};\text{ }t_{n}+c_{i}h_{n}%
\mathbf{,}\text{ }h_{n}\sum_{j=1}^{i-1}a_{ij}\mathbf{k}_{j}).
\end{equation*}
This suggests the following definition.

\begin{definition}
(\cite{de la Cruz 07}) An order $\gamma$ \textit{Local Linear-Runge Kutta
(LLRK) discretization }is an order $\gamma$ Local Linear discretization of
the form (\ref{ODE-HLLA-1}), where the approximation $\mathbf{r}_{\kappa}$
to the remainder term (\ref{ODE-HLLA-2}) is defined by the \textit{Runge
Kutta} formula (\ref{RK}).
\end{definition}

\section{Convergence and linear stability}

In order to study the rate of convergence of the LLRK discretizations, three
useful lemmas will be stated first.

\begin{lemma}
\label{LLRK local error gen}Let $\mathbf{u}_{n+1}=\mathbf{u}_{n}+{\Lambda }^{%
\mathbf{y}_{n}}\left( t_{n},\mathbf{u}_{n};h_{n}\right) $ be an approximate
solution of the auxiliary equation (\ref{Diff. Equat for rn})-(\ref{Initial
Cond. Diff. Equat for rn}) at $t=t_{n+1}\in \left( t\right) _{h}$ given by
an order $\gamma $ numerical integrator, and $\mathbf{y}_{n+1}$ the
discretization
\begin{equation*}
\mathbf{y}_{n+1}=\mathbf{y}_{n}+h_{n}\mathbf{\digamma }(t_{n},\mathbf{y}%
_{n};h_{n}),
\end{equation*}%
where
\begin{equation*}
\mathbf{\digamma }(s,\mathbf{\xi };h)=\frac{1}{h}\left\{ \mathbf{\phi }(s,%
\mathbf{\xi };h)+{\Lambda }^{\mathbf{\xi }}(s,\mathbf{0};h)\right\}
\end{equation*}%
with $\mathbf{y}_{0}=\mathbf{x}_{0}$. Then the local truncation error $%
L_{n+1}$ satisfies
\begin{equation*}
L_{n+1}=\left\Vert \mathbf{x}(t_{n+1};\mathbf{x}_{0})-\mathbf{x}(t_{n};%
\mathbf{x}_{0})-h_{n}\mathbf{\digamma }(t_{n},\mathbf{x}(t_{n};\mathbf{x}%
_{0});h_{n})\right\Vert \leq C_{1}(\mathbf{x}_{0})h_{n}^{\gamma +1}
\end{equation*}%
for all $t_{n},t_{n+1}\in \left( t\right) _{h}$. Moreover, if $\mathbf{%
\digamma }$ satisfies the local Lipschitz condition%
\begin{equation}
\left\Vert \mathbf{\digamma (}s,\mathbf{\xi }_{2};h)-\mathbf{\digamma }(s,%
\mathbf{\xi }_{1};h)\right\Vert \leq B_{\epsilon }\text{ }\left\Vert \mathbf{%
\xi }_{2}-\mathbf{\xi }_{1}\right\Vert \text{, \ \ with }B_{\epsilon }>0%
\text{ and }\mathbf{\xi }_{1},\mathbf{\xi }_{2}\in \epsilon (\mathbf{\xi }%
)\subset \mathcal{D},  \label{Lipschitz}
\end{equation}%
where $\epsilon (\mathbf{\xi })$ is a neighborhood of $\mathbf{\xi }$ for
each $\mathbf{\xi }$ $\subset \mathcal{D}$, then for $h$ small enough there
exists a positive constant $C_{2}(\mathbf{x}_{0})$ depending only on $%
\mathbf{x}_{0}$ such that
\begin{equation*}
\left\Vert \mathbf{x}(t_{n+1};\mathbf{x}_{0})-\mathbf{y}_{n+1}\right\Vert
\leq C_{2}(\mathbf{x}_{0})h^{\gamma }
\end{equation*}%
for all $t_{n+1}\in \left( t\right) _{h}$.
\end{lemma}

\begin{proof}
Taking into account that
\begin{equation*}
\mathbf{x}(t_{n+1};\mathbf{x}_{0})=\mathbf{y}_{LL}\left( t_{n}+h_{n};t_{n},%
\mathbf{x}(t_{n};\mathbf{x}_{0})\right) +\mathbf{r}\left( t_{n}+h_{n};t_{n},%
\mathbf{x}(t_{n};\mathbf{x}_{0})\right) ,
\end{equation*}%
where $\mathbf{y}_{LL}$ and $\mathbf{r}$ are defined as in (\ref{ODE_HLLA-2b}%
) and (\ref{ODE-HLLA-2}), respectively, it is obtained that
\begin{equation*}
L_{n+1}=\left\Vert \mathbf{r}\left( t_{n}+h_{n};t_{n},\mathbf{x}(t_{n};%
\mathbf{x}_{0})\right) -{\Lambda }^{\mathbf{x}(t_{n};\mathbf{x}_{0})}\mathbf{%
(}t_{n},\mathbf{0};h_{n})\right\Vert ,
\end{equation*}%
where $L_{n+1}$ denotes the local truncation error of the discretization
under consideration. Since $\mathbf{r}\left( t_{n}+h_{n};t_{n},\mathbf{x}%
(t_{n};\mathbf{x}_{0})\right) $ is the exact solution of the equation (\ref%
{Diff. Equat for rn})-(\ref{Initial Cond. Diff. Equat for rn}) with $\mathbf{%
y}_{n}=\mathbf{x}(t_{n};\mathbf{x}_{0})$ at $t_{n+1}$ and $\mathbf{u}_{n+1}=%
\mathbf{u}_{n}+{\Lambda }^{\mathbf{x}(t_{n};\mathbf{x}_{0})}\left( t_{n},%
\mathbf{u}_{n};h_{n}\right) $ is the approximate solution of that equation
at $t_{n+1}$ given by an order $\gamma $ numerical integrator, there exists
a positive constant $C_{1}(\mathbf{x}_{0})$ such that
\begin{equation*}
\left\Vert \mathbf{r}\left( t_{n}+h_{n};t_{n},\mathbf{x}(t_{n};\mathbf{x}%
_{0})\right) -{\Lambda }^{\mathbf{x}(t_{n};\mathbf{x}_{0})}\mathbf{(}t_{n},%
\mathbf{0};h_{n})\right\Vert \leq C_{1}(\mathbf{x}_{0})h_{n}^{\gamma +1},
\end{equation*}%
which provides the stated bound for $L_{n+1}$.

On the other hand, since the compact set $\mathcal{X}=\left\{ \mathbf{x}%
\left( t;\mathbf{x}_{0}\right) :t\in \left[ t_{0},T\right] \right\} $ is
contained in the open set $\mathcal{D}\subset \mathbb{R}^{d}$, there exists $%
\varepsilon >0$ such that the compact set
\begin{equation*}
\mathcal{A}_{\varepsilon }=\left\{ \xi \in \mathbb{R}^{d}:\underset{\mathbf{x%
}\left( t;\mathbf{x}_{0}\right) \in \mathcal{X}}{\min }\left\Vert \xi -%
\mathbf{x}\left( t;\mathbf{x}_{0}\right) \right\Vert \leq \varepsilon
\right\}
\end{equation*}%
is contained in $\mathcal{D}$. Since $\mathbf{\digamma }$ satisfies the
local Lipschitz condition (\ref{Lipschitz}), Lemma 2 in \cite{Perko01} (pp.
92) implies the existence of a positive constant $L$ such that%
\begin{equation}
\left\Vert \mathbf{\digamma (}s,\mathbf{\xi }_{2};h)-\mathbf{\digamma }(s,%
\mathbf{\xi }_{1};h)\right\Vert \leq L\text{ }\left\Vert \mathbf{\xi }_{2}-%
\mathbf{\xi }_{1}\right\Vert  \label{Global Lipschitz}
\end{equation}%
for all $\mathbf{\xi }_{1},\mathbf{\xi }_{2}\in \mathcal{A}_{\varepsilon }$.
Hence, the stated estimate $\left\Vert \mathbf{x}(t_{n+1};%
\mathbf{x}_{0})-\mathbf{y}_{n+1}\right\Vert \leq C_{2}(\mathbf{x}%
_{0})h^{\gamma }$ for the global error straightforwardly follows
from the Lipschitz condition (\ref{Global Lipschitz}) and Theorem
3.6 in \cite{Hairer-Wanner93}, where $C_{2}(\mathbf{x}_{0})$ is a
positive contant. Finally, in order to guarantee that
$\mathbf{y}_{n+1}\in \mathcal{A}_{\varepsilon }$ for all
$n=0,...,N-1,$ and so that the LLRK
discretization is well-defined, it is sufficient that $0<h<\delta $, where $%
\delta $ is chosen in such a way that $C_{2}(\mathbf{x}_{0})\delta ^{\gamma
}\leq \varepsilon $.
\end{proof}

Note that this lemma requires of an order $\gamma$ numerical integrator for
the auxiliary equation (\ref{Diff. Equat for rn})-(\ref{Initial Cond. Diff.
Equat for rn}). For this, certain conditions on the vector field $\mathbf{q}$
of this equation have to be assumed (usually, Lipschitz and smoothness
conditions). The next two lemmas show that the function $\mathbf{\phi}$%
\textbf{,} and so the vector field $\mathbf{q}$\textbf{,} inherits such
conditions from the vector field $\mathbf{f}$.

\begin{lemma}
\label{Lemma de phi LLT} Let $\mathbf{\varphi}(.;h)=\frac{1}{h}\mathbf{\phi }%
\left( .;h\right) $. Suppose that
\begin{equation*}
\mathbf{f\in}\text{ }\mathcal{C}^{p+1,q+1}\left( [t_{0},T]\times \mathcal{D},%
\mathbb{R}^{d}\right) ,
\end{equation*}
where $p,q\in\mathbb{N}$. Then $\mathbf{\varphi}\in\mathcal{C}%
^{p,q,r}([t_{0},T]\times\mathcal{D}\times\mathbb{R}_{+},\mathbb{R}^{d})$ for
all $r\in\mathbb{N}$.
\end{lemma}

\begin{proof}
Let $\vartheta _{j}$ be the analytical function recursively defined by
\begin{equation*}
\vartheta _{j+1}\left( z\right) =\left\{
\begin{array}{c}
\left( \vartheta _{j}\left( z\right) -1/j!\right) /z \\
e^{z}%
\end{array}%
\begin{array}{c}
\text{ \ \ \ for }j=1,2,\ldots \\
j=0%
\end{array}%
\right\} \text{ \ }
\end{equation*}%
for $z\in \mathbb{C}$. Since
\begin{equation*}
\vartheta _{j}\left( s\mathbf{M}\right) =\frac{1}{\left( j-1\right) !s^{j}}%
\int_{0}^{s}e^{\left( s-u\right) \mathbf{M}}u^{j-1}du,
\end{equation*}%
for all $s\in \mathbb{R}_{+}$ and $\mathbf{M}\in \mathbb{R}^{d}\times
\mathbb{R}^{d}$ (see for instance \cite{Sidje 1998}), the function $\mathbf{%
\varphi }$ can be written as
\begin{equation*}
\mathbf{\varphi }\left( \tau ,\mathbf{\xi };\delta \right) =\vartheta
_{1}\left( \delta \mathbf{f}_{\mathbf{x}}\left( \tau ,\mathbf{\xi }\right)
\right) \mathbf{f}\left( \tau ,\mathbf{\xi }\right) +\vartheta _{2}\left(
\delta \mathbf{f}_{\mathbf{x}}\left( \tau ,\mathbf{\xi }\right) \right)
\mathbf{f}_{t}\left( \tau ,\mathbf{\xi }\right) \delta
\end{equation*}%
for all $\tau \in \mathbb{R}$, $\mathbf{\xi }\in \mathbb{R}^{d}$ and $\delta
\geq 0$. Thus, from the analyticity of $\vartheta _{j}$ and the continuity
of $\mathbf{f}$ the proof is completed.
\end{proof}

\begin{lemma}
\label{Lemma for Local Error LLRK} Let $\mathbf{f}$ and $\mathbf{q}$ be the
vector fields of the ODEs (\ref{ODE-LLA-1})-(\ref{ODE-LLA-2}) and (\ref%
{Diff. Equat for rn})-(\ref{Initial Cond. Diff. Equat for rn}), respectively.
\end{lemma}

\begin{enumerate}
\item[\textit{i)}] There exists $\varepsilon >0$ such that the compact set
\begin{equation*}
\mathcal{A}_{\varepsilon }=\left\{ \mathbf{z}\in \mathbb{R}^{d}:\underset{%
t\in \lbrack t_{0},T]}{\min }\left\Vert \mathbf{x}\left( t\right) -\mathbf{z}%
\right\Vert \leq \varepsilon \right\}
\end{equation*}%
is contained in $\mathcal{D}$. Moreover, there exists a compact set $%
\mathcal{K}_{\varepsilon }$ included into an open neiborhood of $\mathbf{0}$
and a $\delta _{\varepsilon }>0$, such that
\begin{equation*}
\mathbf{x}(t)+\mathbf{\phi }\left( t,\mathbf{x}(t);\delta \right) +\mathbf{%
\xi }\in \mathcal{A}_{\varepsilon },
\end{equation*}%
for all $\delta \in \lbrack 0,\delta _{\varepsilon }]$, $\mathbf{\xi }\in
\mathcal{K}_{\varepsilon }$ and $t\in \lbrack t_{0},T]$.

\item[\textit{ii)}] If $\mathbf{f}$ and its first partial derivatives are
bounded on $[t_{0},T]\times \mathcal{D}$, and $\mathbf{f}(t,.)$ is a locally
Lipschitz function on $\mathcal{D}$ with Lipschitz constant independent of $%
t $, then there exists a positive constant $P$ such that%
\begin{equation*}
\left\Vert \mathbf{q(}t,\mathbf{x}(t);t+\delta ,\mathbf{\xi }_{2}\mathbf{)-q(%
}t,\mathbf{x}(t);t+\delta ,\mathbf{\xi }_{1}\mathbf{)}\right\Vert \leq
P\left\Vert \mathbf{\xi }_{2}-\mathbf{\xi }_{1}\right\Vert
\end{equation*}%
for all $\delta \in \lbrack 0,\delta _{\varepsilon }]$, $\mathbf{\xi }_{1},%
\mathbf{\xi }_{2}\in \mathcal{K}_{\varepsilon }$ and $t\in \lbrack t_{0},T]$.

\item[\textit{iii)}] If $\mathbf{f}\in \mathcal{C}^{p}\left( [t_{0},T]\times
\mathcal{D},\mathbb{R}^{d}\right) $ for some $p\in \mathbb{N}$, then $%
\mathbf{q(}t,\mathbf{x}(t);\cdot \mathbf{)}\in \mathcal{C}^{p}([t,t+\delta
_{\varepsilon }]\times \mathcal{K}_{\varepsilon },\mathbb{R}^{d})$ for all $%
t\in \lbrack t_{0},T]$.
\end{enumerate}

\begin{proof}
The first part of assertion \textit{i)} follows from the fact that $\mathcal{%
X}=\left\{ \mathbf{x}\left( t\right) :t\in \left[ t_{0},T\right] \right\} $
is a compact set contained into the open set $\mathcal{D}$, whereas its
second part results from the continuity of $\mathbf{\phi }$ on $%
[t_{0},T]\times \mathcal{A}_{\varepsilon }\times \lbrack 0,\delta
_{\varepsilon }]$ stated by the Lemma \ref{Lemma de phi LLT}. Assertion
\textit{ii)} is a straighforward consecuence of Lemma 2 in \cite{Perko01}
(pp. 92). Assertion \textit{iii)} follows from the definition of the vector
field $\mathbf{q}$ and Lemma \ref{Lemma de phi LLT}.
\end{proof}

The next theorem characterizes the convergence rate of LLRK discretizations.
For this purpose, for all $t_{n}\in\left( t\right) _{h}$, denote by
\begin{equation}
\mathbf{y}_{n+1}=\mathbf{y}_{n}+h_{n}\mathbf{\varphi}_{\gamma}(t_{n},\mathbf{%
y}_{n};h_{n})  \label{LLRK_Discretizat}
\end{equation}
the LL discretization defined in (\ref{ODE-HLLA-1}), taking $\mathbf{r}%
_{\kappa}$ as an order $\gamma$ RK scheme of the form (\ref{RK}). That is,
\begin{equation*}
\mathbf{\varphi}_{\gamma}(t_{n},\mathbf{y}_{n};h_{n})=\frac{1}{h_{n}}\left\{
\mathbf{\phi}\left( t_{n},\mathbf{y}_{n};h_{n}\right) +\mathbf{\rho}\left(
t_{n},\mathbf{y}_{n};h_{n}\right) \right\} ,
\end{equation*}
where $\mathbf{\phi}$ is defined by (\ref{ODE-LLA-7}).

\begin{theorem}
\label{Local Error LLRK}Suppose that%
\begin{equation}
\mathbf{f}\in \mathcal{C}^{\gamma +1}([t_{0},T]\times \mathcal{D},\mathbb{R}%
^{d}).  \label{ODE-CONV-8}
\end{equation}%
Then
\begin{equation*}
\left\Vert \mathbf{x}(t_{n}+h;\mathbf{x}_{0})-\mathbf{x}(t_{n};\mathbf{x}%
_{0})-h\mathbf{\varphi }_{\gamma }(t_{n},\mathbf{x}(t_{n};\mathbf{x}%
_{0});h)\right\Vert \leq C_{1}(\mathbf{x}_{0})h^{\gamma +1},
\end{equation*}%
and the LLRK discretization (\ref{LLRK_Discretizat}) satisfies
\begin{equation*}
\left\Vert \mathbf{x}(t_{n+1};\mathbf{x}_{0})-\mathbf{y}_{n+1}\right\Vert
\leq C_{2}(\mathbf{x}_{0})h^{\gamma },
\end{equation*}%
for all $t_{n},t_{n+1}\in \left( t\right) _{h}$, where $C_{1}(\mathbf{x}%
_{0}) $ and $C_{2}(\mathbf{x}_{0})$ are positive constants depending only on
$\mathbf{x}_{0}$.
\end{theorem}

\begin{proof}
By Theorem 3.1 in \cite{Hairer-Wanner93}, the local truncation error of the
order $\gamma$\ explicit RK scheme (\ref{RK}) for the equation (\ref{Diff.
Equat for rn})-(\ref{Initial Cond. Diff. Equat for rn}) with $\mathbf{y}_{n}=%
\mathbf{x}(t_{n};\mathbf{x}_{0})$ is
\begin{equation}
\left\Vert \mathbf{u}(t_{n}+h)-\mathbf{\rho}\left( t_{n},\mathbf{x}(t_{n};%
\mathbf{x}_{0});h\right) \right\Vert \leq C(\mathbf{x}_{0})\text{ }%
h^{\gamma+1},  \label{ODE-CONV-9}
\end{equation}
where
\begin{equation*}
C(\mathbf{x}_{0})=\frac{1}{(\gamma+1)!}\underset{\theta\in\lbrack0,1]}{\max }%
\left\Vert \frac{d^{\gamma+1}}{dt^{\gamma+1}}\mathbf{u}(t_{n}+\theta
h)\right\Vert +\frac{1}{\gamma!}\sum\limits_{i=1}^{s}\left\vert
b_{i}\right\vert \underset{\theta\in\lbrack0,1]}{\max}\left\Vert \frac{%
d^{\gamma}}{dt^{\gamma}}\mathbf{k}_{i}(\theta h)\right\Vert
\end{equation*}
with
\begin{equation*}
\mathbf{k}_{i}(\theta h)=\mathbf{q}\left( t_{n},\mathbf{x}(t_{n};\mathbf{x}%
_{0}),t_{n}\mathbf{+}c_{i}\theta h\mathbf{,}\theta h\sum _{j=1}^{i-1}a_{ij}%
\mathbf{k}_{j}(\theta h)\right) .
\end{equation*}

By taking into account that the solution $\mathbf{r}$ of (\ref{Diff. Equat
for rn})-(\ref{Initial Cond. Diff. Equat for rn}) is the remainder term of
the LL approximation and by setting $\mathbf{y}_{n}=\mathbf{x}(t_{n};\mathbf{%
x}_{0})$ in (\ref{Diff. Equat for rn}), it follows that
\begin{equation}
\mathbf{u}\left( t_{n}+\theta h\right) =\mathbf{x}\left( t_{n}+\theta h;%
\mathbf{x}_{0}\right) -\mathbf{x}\left( t_{n};\mathbf{x}_{0}\right) -\mathbf{%
\phi }\left( t_{n},\mathbf{x}\left( t_{n};\mathbf{x}_{0}\right) ;\theta
h\right) ,  \label{equat_u}
\end{equation}%
and so%
\begin{equation*}
\left\Vert \frac{d^{\gamma +1}}{dt^{\gamma +1}}\mathbf{u}(t_{n}+\theta
h)\right\Vert =\left\Vert \frac{d^{\gamma }}{dt^{\gamma }}\mathbf{q}(t_{n,}%
\mathbf{x}(t_{n};\mathbf{x}_{0});t_{n}+\theta h,\mathbf{u}(t_{n}+\theta
h))\right\Vert ,
\end{equation*}%
where the derivative in the right term of the last expression is with
respect to the last two arguments of the function $\mathbf{q}$. Condition (%
\ref{ODE-CONV-8}), assertion $iii$) of Lemma \ref{Lemma for Local Error LLRK}
and expression (\ref{equat_u}) imply that $\mathbf{q(}.\mathbf{,x}(.,\mathbf{%
x}_{0});.,\mathbf{u}(.))\in \mathcal{C}^{\gamma }([t_{0},T],\mathbb{R}^{d})$%
. Hence, there exists a constant $M$ such that
\begin{equation*}
\underset{\theta \in \lbrack 0,1],\text{ }t_{n}\in \lbrack t_{0},T]}{\max }%
\left\Vert \frac{d^{\gamma +1}}{dt^{\gamma +1}}\mathbf{u}(t_{n}+\theta
h)\right\Vert \leq M.
\end{equation*}%
Likewise, condition (\ref{ODE-CONV-8}) and Lemma \ref{Lemma for Local Error
LLRK} imply that
\begin{equation*}
\underset{\theta \in \lbrack 0,1],\text{ }t_{n}\in \lbrack t_{0},T]}{\max }%
\left\Vert \frac{d^{\gamma }}{dt^{\gamma }}\mathbf{k}_{i}(\theta
h)\right\Vert \leq M.
\end{equation*}%
Therefore, $C(\mathbf{x}_{0})$ in (\ref{ODE-CONV-9}) is bounded as a
function of $\mathbf{x}_{0}\in \mathcal{D}$.

In addition, Lemma \ref{Lemma de phi LLT} and Lemma 3.5 in \cite%
{Hairer-Wanner93} combined with assertion $iii)$ of Lemma \ref{Lemma for
Local Error LLRK} imply that $\mathbf{\phi }$ and $\mathbf{\rho }$ satisfy
the local Lipschitz condition (\ref{Lipschitz}), and so does the function
\begin{equation*}
\mathbf{\varphi }_{\gamma }(t_{n},\mathbf{y}_{n};h)=\frac{1}{h}\left\{
\mathbf{\phi }\left( t_{n},\mathbf{y}_{n};h\right) +\mathbf{\rho }\left(
t_{n},\mathbf{y}_{n};h\right) \right\}
\end{equation*}%
as well. This and Lemma \ref{LLRK local error gen} complete the proof.
\end{proof}

Note that the Lipschitz and smoothness conditions in Lemma \ref{LLRK local
error gen} and Theorem \ref{Local Error LLRK} are the usual ones required to
derive the convergence of numerical integrators (see, e.g., Theorems 3.1 and
3.6 in \cite{Hairer-Wanner93}). These conditions directly imply that
smoothness of the solution of the ODE in a bounded domain (see, e.g.,
Theorem 1 pp. 79 and Remark 1 pp. 83 in \cite{Perko01}). In this way, to
ensure the convergence of the LLRK integrators, the involved RK coefficients
are not constrained by any stability condition and they just need to satisfy
the usual order conditions for RK schemes. This is a major difference with
the Rosenbrock and Exponential Integrators and makes the LLRK methods more
flexible and simple. Further note that, like these integrators, the LLRK are
trivially A-stable.

\section{Steady states \label{Sec. teoria}}

In this section the relation between the steady states of an autonomous
equation%
\begin{align}
\frac{d\mathbf{x}\left( t\right) }{dt} & =\mathbf{f}\left( \mathbf{x}\left(
t\right) \right) \text{, \ \ }t\in\left[ t_{0},T\right] ,  \label{ODE-SS-0a}
\\
\mathbf{x}(t_{0}) & =\mathbf{x}_{0}\in\mathbb{R}^{d},  \label{ODE-SS-0b}
\end{align}
and those of their LLRK discretizations is considered. For the sake of
simplicity, a uniform time partition $h_{n}=h$ is adopted.

It will be convenient to rewrite the order $\gamma $ LLRK discretization in
the form%
\begin{equation}
\mathbf{y}_{n+1}=\mathbf{y}_{n}+h\mathbf{\varphi }_{\gamma }(\mathbf{y}%
_{n},h),  \label{ODE-SS-1}
\end{equation}%
where
\begin{equation}
\mathbf{\varphi }_{\gamma }\left( \mathbf{\xi ,}\delta \right) ={\Phi }(%
\mathbf{\xi },\delta )\mathbf{f}(\mathbf{\xi })+\sum_{i=1}^{s}b_{i}\mathbf{k}%
_{i}\left( \mathbf{\xi ,}\delta \right) ,  \label{ODE-SS-9}
\end{equation}%
with
\begin{equation}
{\Phi }(\mathbf{\xi },\delta )=\frac{1}{\delta }\int\limits_{0}^{\mathbb{%
\delta }}e^{\mathbf{f}_{\mathbf{x}}(\mathbf{\xi })u}du,  \label{ODE-SS-3}
\end{equation}%
\begin{equation*}
\mathbf{k}_{i}\left( \mathbf{\xi ,}\delta \right) =\mathbf{q}(\mathbf{\xi };%
\text{ }c_{i}\delta ,\text{ }\delta \sum_{j=1}^{i-1}a_{ij}\mathbf{k}%
_{j}\left( \mathbf{\xi ,}\delta \right) )
\end{equation*}%
and
\begin{equation*}
\mathbf{q(\xi };\delta \mathbf{,u)}=\mathbf{f(\xi }+\delta {\Phi }(\mathbf{%
\xi },\delta )\mathbf{f}(\mathbf{\xi })+\mathbf{u})-\mathbf{f}_{\mathbf{x}}(%
\mathbf{\xi })\delta {\Phi }(\mathbf{\xi },\delta )\mathbf{f}(\mathbf{\xi })-%
\mathbf{f}\left( \mathbf{\xi }\right) .
\end{equation*}

For later reference, the following Lemma states some useful properties of
the functions $\mathbf{\varphi}_{\gamma}$ on neighborhoods of invariant sets
of ODEs.

\begin{lemma}
\label{Lemma 5.4}Let $\Sigma\subset\mathbb{R}^{d}$ be an invariant set for
the flow of the equation (\ref{ODE-SS-0a}). Let $\mathcal{K}$ and $\Omega$
be, respectively, compact and bounded open sets such that $\Sigma\subset
\mathcal{K}\subset\Omega$. Suppose that the solution $\mathbf{x}$ of (\ref%
{ODE-SS-0a})-(\ref{ODE-SS-0b}) fulfils the condition%
\begin{equation}
\mathbf{x}(t;\mathbf{x}_{0})\subset\Omega\text{ for all initial point }%
\mathbf{x}_{0}\in\mathcal{K}\text{ and }t\in\lbrack t_{0},T],  \label{H1}
\end{equation}
and the vector field $\mathbf{f}$ satisfies the continuity condition%
\begin{equation}
\mathbf{f}\in\mathcal{C}^{\gamma+1}(\Omega,\mathbb{R}^{d}).  \label{H2}
\end{equation}
Further, let%
\begin{equation*}
\mathbf{y}_{n+1}=\mathbf{y}_{n}+h\mathbf{\varphi}_{\gamma}(\mathbf{y}_{n},h)
\end{equation*}
be the order $\gamma$ LLRK discretization defined by (\ref{ODE-SS-1}). Then

\begin{enumerate}
\item[\textit{i)}] $\mathbf{\varphi}_{\gamma}\rightarrow\mathbf{f}$ and $%
\partial\mathbf{\varphi}_{\gamma}\mathbf{/\partial y}_{n}\rightarrow \mathbf{%
f}_{\mathbf{x}}$ \textit{as }$h\rightarrow0$\textit{\ uniformly in }$%
\mathcal{K}$\textit{,}

\item[\textit{ii)}] $\left\Vert (\mathbf{x}(t_{0}+h;\mathbf{x}_{0})-\mathbf{x%
}_{0})/h-\mathbf{\varphi }_{\gamma }(\mathbf{x}_{0},h)\right\Vert
=O(h^{\gamma })$\textit{\ uniformly for }$\mathbf{x}_{0}\in \mathcal{K}$.
\end{enumerate}
\end{lemma}

\begin{proof}
According to Lemma 5 in \cite{Jimenez02 AMC}, $\mathbf{f}\in\mathcal{C}%
^{\gamma+1}(\Omega)$ implies that ${\Phi f}\rightarrow\mathbf{f}$ and $%
\partial({\Phi f)/\partial\xi}\rightarrow\mathbf{f}_{\mathbf{x}}$ as $%
h\rightarrow0$\ uniformly in $\mathcal{K}$.

On the other hand, $\mathbf{k}_{i}(\mathbf{\xi },0)=\mathbf{0}$, for all $%
\mathbf{\xi }\in \Omega $ and $i=1,\ldots ,s$. Besides, since
\begin{equation*}
\frac{\partial \mathbf{k}_{i}}{\mathbf{\partial \xi }}(\mathbf{\xi },\delta
)=\frac{\partial \mathbf{q}}{\mathbf{\partial \xi }}(\mathbf{\xi };\text{ }%
c_{i}\delta ,\delta \sum_{j=1}^{i-1}a_{ij}\mathbf{k}_{j}(\mathbf{\xi }%
,\delta )),
\end{equation*}%
where
\begin{align*}
\frac{\partial \mathbf{q}}{\mathbf{\partial \xi }}\mathbf{(\xi };\delta
\mathbf{,u)}& =\mathbf{\mathbf{f}_{\mathbf{x}}(\xi }+\delta {\Phi }(\mathbf{%
\xi },\delta )\mathbf{f}(\mathbf{\xi })+\mathbf{u})\text{ }\frac{\partial }{%
\mathbf{\partial \xi }}\left( \mathbf{\xi }+\delta {\Phi }(\mathbf{\xi }%
,\delta )\mathbf{f}(\mathbf{\xi })+\mathbf{u}\right) \\
& -\delta \frac{\partial }{\mathbf{\partial \xi }}\left( \mathbf{f}_{\mathbf{%
x}}(\mathbf{\xi }){\Phi }(\mathbf{\xi },\delta )\mathbf{f}(\mathbf{\xi }%
)\right) -\mathbf{f}_{\mathbf{x}}\left( \mathbf{\xi }\right) +\mathbf{%
\mathbf{f}_{\mathbf{x}}(\xi }+\delta {\Phi }(\mathbf{\xi },\delta )\mathbf{f}%
(\mathbf{\xi })+\mathbf{u})\text{ }\frac{\partial \mathbf{u}}{\mathbf{%
\partial \xi }}
\end{align*}%
with
\begin{equation*}
\mathbf{u}=\delta \sum_{j=1}^{i-1}a_{ij}\mathbf{k}_{j}(\mathbf{\xi },\delta )%
\text{ \ \ \ \ \ and \ \ \ \ }\frac{\partial \mathbf{u}}{\mathbf{\partial
\xi }}=\delta \sum_{j=1}^{i-1}a_{ij}\frac{\partial }{\mathbf{\partial \xi }}%
\mathbf{k}_{j}(\mathbf{\xi },\delta ),
\end{equation*}%
$\partial \mathbf{k}_{i}(\mathbf{\xi },0)\mathbf{/\partial \xi }=\mathbf{0}$
for all $i=1,\ldots ,s$. Thus, since each $\mathbf{k}_{i}$ and $\partial
\mathbf{k}_{i}\mathbf{/\partial \xi }$ are continuous functions on $\Omega
\times \lbrack 0,1]$, it holds that $\mathbf{k}_{i}\rightarrow \mathbf{0}$
and $\partial \mathbf{k}_{i}\mathbf{/\partial \xi }\rightarrow \mathbf{0}$
as $h\rightarrow 0$\ uniformly in the compact set $\mathcal{K}$. Thus,
assertion \textit{i)} holds.

From Theorem \ref{Local Error LLRK} we have
\begin{equation*}
\left\Vert (\mathbf{x}(t_{0}+h;\mathbf{x}_{0})-\mathbf{x}_{0})/h-\mathbf{%
\varphi }_{\gamma }(\mathbf{x}_{0},h)\right\Vert \leq C(\mathbf{x}%
_{0})h^{\gamma },
\end{equation*}%
where%
\begin{equation*}
C(\mathbf{x}_{0})=\frac{1}{(\gamma +1)!}\underset{\theta \in \lbrack 0,1]}{%
\max }\left\Vert \frac{d^{\gamma +1}}{dt^{\gamma +1}}\mathbf{u}(t_{0}+\theta
h)\right\Vert +\frac{1}{\gamma !}\sum\limits_{i=1}^{s}\left\vert
b_{i}\right\vert \underset{\theta \in \lbrack 0,1]}{\max }\left\Vert \frac{%
d^{\gamma }}{dt^{\gamma }}\mathbf{k}_{i}(\theta h)\right\Vert
\end{equation*}%
is a positive constant depending of $\mathbf{x}_{0}$,
\begin{equation*}
\mathbf{k}_{i}(\theta h)=\mathbf{q}\left( \mathbf{x}(t_{0};\mathbf{x}%
_{0});c_{i}\theta h\mathbf{,}\theta h\sum_{j=1}^{i-1}a_{ij}\mathbf{k}%
_{j}(\theta h)\right) ,\text{\ \ \ \ \ \ \ \ }i=1,\ldots ,s.
\end{equation*}%
and $\mathbf{u}\left( t_{0}+\theta h\right) =\mathbf{x}\left( t_{0}+\theta h;%
\mathbf{x}_{0}\right) -\mathbf{x}\left( t_{0};\mathbf{x}_{0}\right) -\mathbf{%
\phi }\left( t_{0},\mathbf{x}\left( t_{0};\mathbf{x}_{0}\right) ;\theta
h\right) $.

Clearly,
\begin{eqnarray*}
\left\Vert \frac{d^{\gamma +1}}{dt^{\gamma +1}}\mathbf{u}(s)\right\Vert
&=&\left\Vert \frac{d^{\gamma +1}}{dt^{\gamma +1}}(\mathbf{x}\left( s;%
\mathbf{x}_{0}\right) -\mathbf{x}\left( t_{0};\mathbf{x}_{0}\right) -\mathbf{%
\phi }\left( t_{0},\mathbf{x}\left( t_{0};\mathbf{x}_{0}\right)
;s-t_{0}\right) )\right\Vert \\
&\leq &\left\Vert \frac{d^{\gamma }}{dt^{\gamma }}\mathbf{f}(\mathbf{x}%
\left( s;\mathbf{x}_{0}\right) )\right\Vert +\left\Vert \frac{d^{\gamma +1}}{%
dt^{\gamma +1}}\mathbf{\phi }\left( t_{0},\mathbf{x}\left( t_{0};\mathbf{x}%
_{0}\right) ;s-t_{0}\right) \right\Vert
\end{eqnarray*}%
for all $s\in \lbrack t_{0},t_{0}+h]$. Since $\mathbf{x}(t;\mathbf{x}%
_{0})\in \Omega $ for all $t\in \lbrack t_{0},T]$ and $\mathbf{x}_{0}\in
\mathcal{K}$ $\subset $ $\Omega $, there exists a compact set $\mathcal{A}%
_{h}$ depending of $h$ such that $\mathcal{K}\subset \mathcal{A}_{h}$ $%
\subset \Omega $ and $\mathbf{\mathbf{x}}(s\mathbf{;\mathbf{x}}_{0})\in
\mathcal{A}_{h}$ for all $s\in \lbrack t_{0},t_{0}+h]$\ and $\mathbf{\mathbf{%
x}}_{0}\in \mathcal{K}$. In addition, since condition $\mathbf{f}\in
\mathcal{C}^{\gamma +1}(\Omega ,\mathbb{R}^{d})$ implies that there exists a
constant $M$ such that
\begin{equation*}
\underset{\xi \in \mathcal{A}_{h}}{\sup }\left\Vert \frac{d^{\gamma }}{%
dt^{\gamma }}\mathbf{f}(\xi )\right\Vert \leq M,
\end{equation*}%
it is obtained that%
\begin{equation*}
\underset{\theta \in \lbrack 0,1],\text{ }\mathbf{\mathbf{x}}_{0}\in
\mathcal{K}}{\max }\left\Vert \frac{d^{\gamma }}{dt^{\gamma }}\mathbf{f}(%
\mathbf{x}\left( t_{0}+\theta h;\mathbf{x}_{0}\right) )\right\Vert \leq
\underset{\xi \in \mathcal{A}_{h}}{\sup }\left\Vert \frac{d^{\gamma }}{%
dt^{\gamma }}\mathbf{f}(\xi )\right\Vert \leq M.
\end{equation*}

Taking into account that $\mathbf{\phi }$ and $\mathbf{k}_{i}$ are functions
of $\mathbf{f}$, we can similarly proceed to find a bound $B>0$ independent
of $\theta $, $\mathbf{x}_{0}$ for $\left\Vert \frac{d^{\gamma +1}}{%
dt^{\gamma +1}}\mathbf{\phi }\left( t_{0},\mathbf{x}\left( t_{0};\mathbf{x}%
_{0}\right) ;s-t_{0}\right) \right\Vert $ and $\left\Vert \frac{d^{\gamma }}{%
dt^{\gamma }}\mathbf{k}_{i}(\theta h)\right\Vert $. Hence, we conclude that $%
C(\mathbf{x}_{0})$ is bounded on $\mathcal{K}$ by a constant independent of $%
\mathbf{x}_{0}$, and so \textit{ii)} follows.
\end{proof}

\subsection{Fixed points and linearization preserving}

\begin{theorem}
\label{Prop. 5.2 LLRK} Suppose that the vector field $\mathbf{f}$ of the
equation (\ref{ODE-SS-0a}) and its derivatives up to order $\gamma $\ are
defined and bounded on $\mathbb{R}^{d}$. Then, all equilibrium points of the
given ODE (\ref{ODE-SS-0a}) are fixed points of any LLRK discretization.
\end{theorem}

\begin{proof}
Let $\mathbf{\varphi }_{\gamma }$, ${\Phi }$ and $\mathbf{k}_{i}$ be the
functions defined in expression (\ref{ODE-SS-9}). If $\mathbf{\xi }$ is an
equilibrium point of (\ref{ODE-SS-0a}), then $\mathbf{f}(\mathbf{\xi })=%
\mathbf{0}$ and so ${\Phi }(\mathbf{\xi },h)\mathbf{f}(\mathbf{\xi })=%
\mathbf{0}$ and $\mathbf{k}_{i}(\mathbf{\xi },h)=\mathbf{0}$ for all $h$ and
$i=1,\ldots ,s$. Thus, $\mathbf{\varphi }_{\gamma }(\mathbf{\xi },h)=\mathbf{%
0}$ for all $h$, which implies that $\mathbf{\xi }$ is a fixed point of the
LLRK discretization (\ref{ODE-SS-1}).
\end{proof}

A numerical integrator $\mathbf{u}_{n+1}=\mathbf{u}_{n}+{\Lambda }\left(
t_{n},\mathbf{u}_{n};h_{n}\right) $ is linearization preserving at an
equilibrium point $\mathbf{\xi }$ of the ODE (\ref{ODE-SS-0a}) if from the
Taylor series expansion of ${\Lambda }\left( t_{n},\mathbf{\cdot }%
;h_{n}\right) $ around $\mathbf{\xi }$ it is obtained that%
\begin{equation*}
\mathbf{u}_{n+1}-\mathbf{\xi }=e^{h\mathbf{f}_{x}(\mathbf{\xi })}(\mathbf{u}%
_{n}-\mathbf{\xi )}+O(\left\Vert \mathbf{u}_{n}-\mathbf{\xi }\right\Vert
^{2}).
\end{equation*}%
Furthermore, an integrator is said to be linearization preserving if it is
linearization preserving at all equilibrium points of the ODE \cite%
{McLachlan09}.

This property ensures that the integrator correctly captures all eigenvalues
of the linearized system at every equilibrium point of an ODE, which
guarantees the exact preservation (in type and parameters) of a number of
local bifurcations of the underlying equation \cite{McLachlan09}. Certainly,
this results in a correct reproduction of the local dynamics before, during
and after a bifurcation anywhere in the phase space by the numerical
integrator.

In \cite{McLachlan09} the linearization preserving property of the LL
discretization (\ref{ODE-LLA-4}) was demonstrated. This property is also
inherited by LLRK discretizations as it is shown by the next theorem.

\begin{theorem}
Let the vector field $\mathbf{f}$ of the equation (\ref{ODE-SS-0a}) and its
derivatives up to order $2$\ be functions defined and bounded on $\mathbb{R}%
^{d}$. Then, LLRK discretizations are linearization preserving.
\end{theorem}

\begin{proof}
Let $\mathbf{\xi}$ be an arbitrary equilibrium point of the ODE (\ref%
{ODE-SS-0a}) and let the initial condition $\mathbf{y}_{n}$ be in the
neighborhood of $\mathbf{\xi}$.

Let us consider the Taylor expansion of $\mathbf{f}$ around $\mathbf{\xi}$
\begin{equation*}
\mathbf{f}(\mathbf{y}_{n})=\mathbf{f}_{\mathbf{x}}(\mathbf{\xi})(\mathbf{y}%
_{n}-\mathbf{\xi})+O(\left\Vert \mathbf{y}_{n}-\mathbf{\xi}\right\Vert ^{2})
\end{equation*}
and the LL discretization
\begin{equation*}
\mathbf{y}_{n+1}=\mathbf{y}_{n}+h{\Phi}(\mathbf{y}_{n},h)\mathbf{f}(\mathbf{y%
}_{n}),
\end{equation*}
where ${\Phi}$ defined as in (\ref{ODE-SS-3}) is, according to assertion
\textit{i)} of Lemma 1 in \cite{Jimenez02 AMC}, a Lipschitz function. By
combining this Taylor expansion with both, the identity (\ref{ODE-LLA-3})
and the Lipschitz inequality $\left\Vert {\Phi}(\mathbf{y}_{n},h)-{\Phi }(%
\mathbf{\xi},h)\right\Vert \leq$ $\lambda\left\Vert \mathbf{y}_{n}-\mathbf{%
\xi}\right\Vert $ it is obtained
\begin{equation}
h{\Phi}(\mathbf{\xi},h)\mathbf{f}(\mathbf{y}_{n})=(e^{h\mathbf{f}_{\mathbf{x}%
}(\mathbf{\xi})}-\mathbf{I})(\mathbf{y}_{n}-\mathbf{\xi})+O(\left\Vert
\mathbf{y}_{n}-\mathbf{\xi}\right\Vert ^{2})  \label{ODE-SS-11}
\end{equation}
and
\begin{equation}
\left\Vert ({\Phi}(\mathbf{y}_{n},h)-{\Phi}(\mathbf{\xi},h))\mathbf{f}(%
\mathbf{y}_{n})\right\Vert \leq C\left\Vert \mathbf{y}_{n}-\mathbf{\xi }%
\right\Vert ^{2},  \label{ODE-SS-12}
\end{equation}
respectively, where $C$ is a positive constant.

Now, consider the LLRK discretization
\begin{equation*}
\mathbf{y}_{n+1}=\mathbf{y}_{n}+h\mathbf{\varphi }_{\gamma }(\mathbf{y}%
_{n},h),
\end{equation*}%
with $\mathbf{\varphi }_{\gamma }$ defined as in (\ref{ODE-SS-9}). From the
Taylor formula with Lagrange remainder it is obtained that
\begin{align*}
\left\Vert \mathbf{q(y}_{n};h\mathbf{,u)}\right\Vert & =\left\Vert \mathbf{%
f(y}_{n}+{\Phi }(\mathbf{y}_{n},h)\mathbf{f}(\mathbf{y}_{n})h+\mathbf{u})-%
\mathbf{f}_{\mathbf{x}}(\mathbf{y}_{n}){\Phi }(\mathbf{y}_{n},h)\mathbf{f}(%
\mathbf{y}_{n})h-\mathbf{f}\left( \mathbf{y}_{n}\right) \right\Vert \\
& \leq M\left\Vert {\Phi }(\mathbf{y}_{n},h)\mathbf{f}(\mathbf{y}_{n})h+%
\mathbf{u}\right\Vert ^{2}+\left\Vert \mathbf{f}_{\mathbf{x}}(\mathbf{y}%
_{n})\right\Vert \left\Vert \mathbf{u}\right\Vert ,
\end{align*}%
where the positive constant $M$ is a bound for $\left\Vert \mathbf{f}_{%
\mathbf{xx}}\right\Vert $ on a compact subset $\mathcal{K}\subset $ $\mathbb{%
R}^{d}$ such that $\mathbf{y}_{n},\mathbf{y}_{n+1},\mathbf{\xi \in }$ $%
\mathcal{K}$. By using (\ref{ODE-SS-11}) and (\ref{ODE-SS-12}) it follows
that
\begin{equation*}
\left\Vert \mathbf{q(y}_{n};h\mathbf{,u)}\right\Vert \leq 2M\left\Vert
\mathbf{u}\right\Vert ^{2}+\left\Vert \mathbf{f}_{\mathbf{x}}(\mathbf{y}%
_{n})\right\Vert \left\Vert \mathbf{u}\right\Vert +O(\left\Vert \mathbf{y}%
_{n}-\mathbf{\xi }\right\Vert ^{2}).
\end{equation*}%
From the last inequality and taking into account that $\mathbf{k}_{1}=%
\mathbf{0}$, it is obtained that $\left\Vert \mathbf{k}_{2}\right\Vert \leq
O(\left\Vert \mathbf{y}_{n}-\mathbf{\xi }\right\Vert ^{2})$. Furthermore, by
induction, it is obtained that $\left\Vert \mathbf{k}_{i}\right\Vert \leq
O(\left\Vert \mathbf{y}_{n}-\mathbf{\xi }\right\Vert ^{2})$ for all $%
i=1,2,...,s.$ From this, (\ref{ODE-SS-11}) and (\ref{ODE-SS-9}) it follows
that
\begin{equation*}
h\mathbf{\varphi }_{\gamma }(\mathbf{y}_{n},h)=(e^{h\mathbf{f}_{\mathbf{x}}(%
\mathbf{\xi })}-\mathbf{I})(\mathbf{y}_{n}-\mathbf{\xi })+O(\left\Vert
\mathbf{y}_{n}-\mathbf{\xi }\right\Vert ^{2}),
\end{equation*}%
which implies that the LLRK discretization is linearization preserving.
\end{proof}

The next two subsections deal with a more precise analysis of the dynamical
behavior of the LLRK discretizations in the neighborhood of some steady
states.

\subsection{Phase portrait near equilibrium points}

Let $\mathbf{0}$ be a hyperbolic equilibrium point of the equation (\ref%
{ODE-SS-0a}). Let $X_{s},X_{u}\subset\mathbb{R}^{d}$ be the stable and
unstable subspaces of the linear vector field $\mathbf{f}_{\mathbf{x}}(%
\mathbf{0})$ such that $\mathbb{R}^{d}=X_{s}\oplus X_{u}$, $(\mathbf{x}_{s},%
\mathbf{x}_{u})=\mathbf{x}\in\mathbb{R}^{d}$ and $\left\Vert \mathbf{x}%
\right\Vert =\max\{\left\Vert \mathbf{x}_{s}\right\Vert ,\left\Vert \mathbf{x%
}_{u}\right\Vert \}$. It is well-known that the local stable and unstable
manifolds at $\mathbf{0}$ may be represented as $M_{s}=\{(\mathbf{x}_{s},p(%
\mathbf{x}_{s})):\mathbf{x}_{s}\in\mathcal{K}_{\varepsilon,s}\}$ and $%
M_{u}=\{(q(\mathbf{x}_{u}),\mathbf{x}_{u})):\mathbf{x}_{u}\in\mathcal{K}%
_{\varepsilon,u}\}$, respectively, where the functions $p:\mathcal{K}%
_{\varepsilon,s}=\mathcal{K}_{\mathbb{\varepsilon}}\cap X_{s}\rightarrow
\mathcal{K}_{\varepsilon,u}=\mathcal{K}_{\mathbb{\varepsilon}}\cap X_{u}$
and $q:\mathcal{K}_{\varepsilon,u}\rightarrow\mathcal{K}_{\varepsilon,s}$
are as smooth as $\mathbf{f}$, and $\mathcal{K}_{\mathbb{\varepsilon}}=\{%
\mathbf{x}\in\mathbb{R}^{d}:\left\Vert \mathbf{x}\right\Vert
\leq\varepsilon\}$ for $\varepsilon>0$.

\begin{theorem}
\label{Prop. 5.3} Suppose that the conditions (\ref{H1})-(\ref{H2}) of Lemma %
\ref{Lemma 5.4} hold on a neighborhood $\Omega $ of $\mathbf{0}$. Then there
exist constants $C$, $\varepsilon $, $\varepsilon _{0}$, $h_{0}>0$ such that
the local stable $M_{s}^{h}$ and unstable $M_{u}^{h}$ manifolds of the order
$\gamma $ LLRK discretization (\ref{ODE-SS-1}) at $\mathbf{0}$ are of the
form
\begin{equation*}
M_{s}^{h}=\{(\mathbf{x}_{s},p^{h}(\mathbf{x}_{s})):\mathbf{x}_{s}\in
\mathcal{K}_{\varepsilon ,s}\}\text{ and }M_{u}^{h}=\{(q^{h}(\mathbf{x}_{u}),%
\mathbf{x}_{u}):\mathbf{x}_{u}\in \mathcal{K}_{\varepsilon ,u}\},
\end{equation*}%
where $p^{h}=p+O(h^{\gamma })$ uniformly in $\mathcal{K}_{\varepsilon ,s}$ ,
and $q^{h}=q+O(h^{\gamma })$ uniformly in $\mathcal{K}_{\varepsilon ,u}$.
Moreover, for any $\mathbf{x}_{0}\in \mathcal{K}_{\mathbb{\varepsilon }}$
and $h\leq h_{0}$, there exists $\mathbf{z}_{0}=\mathbf{z}_{0}(\mathbf{x}%
_{0},h)\in \mathcal{K}_{\varepsilon _{0}}$ satisfying
\begin{equation}
sup\{\left\Vert \mathbf{x}(t_{n};\mathbf{x}_{0})-\mathbf{y}_{n}(\mathbf{z}%
_{0})\right\Vert :\mathbf{x}(t;\mathbf{x}_{0})\in \mathcal{K}_{\mathbb{%
\varepsilon }}\text{ for }t\in \lbrack t_{0},t_{n}]\}\leq Ch^{\gamma }.
\label{ODE-SS-6}
\end{equation}%
Correspondingly, for any $\mathbf{z}_{0}\in \mathcal{K}_{\mathbb{\varepsilon
}}$ and $h\leq h_{0}$, there exists $\mathbf{x}_{0}=\mathbf{x}_{0}(\mathbf{z}%
_{0},h)\in \mathcal{K}_{\varepsilon _{0}}$ that fulfils (\ref{ODE-SS-6}),
where the supremum is taken over all $n$ satisfying $\mathbf{y}_{j}(\mathbf{z%
}_{0})\in \mathcal{K}_{\mathbb{\varepsilon }}$, $j=0,\ldots ,n$.
\end{theorem}

\begin{proof}
Since $\Omega$ is a neighborhood of the invariant set $\mathbf{0}$, there
exists a constant $\varepsilon>0$ and a compact set $\mathcal{K}%
_{\varepsilon }=\{\mathbf{\xi}\in\mathbb{R}^{d}:\left\Vert \mathbf{\xi}%
\right\Vert \leq\varepsilon\}\subset\Omega$ such that Lemma \ref{Lemma 5.4}
holds with $\mathcal{K}=\mathcal{K}_{\varepsilon}$. Furthermore, by
assertion \textit{i)} of Theorem \ref{Prop. 5.2 LLRK}, $\mathbf{f}(\mathbf{%
\xi})=\mathbf{0}$ implies $\mathbf{\varphi}_{\gamma}(\mathbf{\xi},h)=\mathbf{%
0}$ for all $h$. Thus, the hypotheses of Theorem 3.1 in \cite{Beyn 1987a}
hold for the LLRK discretizations, which completes the prove.
\end{proof}

Theorem \ref{Prop. 5.3} shows that the phase portrait of a continuous
dynamical system near a hyperbolic equilibrium point is correctly reproduced
by LLRK discretizations for sufficiently small step-sizes. It states that
any trajectory of the dynamical system can be correctly approximated by a
trajectory of the LLRK discretization if the discrete initial value is
conveniently adjusted. It also affirms that any trajectory of a LLRK
discretization approximates some trajectory of the continuous system with a
suitably selection of the starting point. In both cases, these results are
valid for sufficiently small step-sizes and as long as the trajectories stay
within some neighborhood of the equilibrium point. Moreover, the theorem
ensures that the local stable and unstable manifolds of a LLRK
discretization at the equilibrium point converge to those of the continuous
system as the step-size goes to zero.

\subsection{Phase portraits near periodic orbits}

Suppose that the equation (\ref{ODE-SS-0a}) has a hyperbolic closed orbit $%
\Gamma =\{\overline{\mathbf{x}}(t):t\in \lbrack 0,T]\}$ of period $T$ in an
open bounded set $\Omega \subset \mathbb{R}^{d}$ . Let $\overline{\Omega }$
be the closure of $\Omega $.

\begin{theorem}
\label{Prop. 5.5} Let the assumptions (\ref{H1})-(\ref{H2}) of Lemma \ref%
{Lemma 5.4} hold on a neighborhood of $\overline{\Omega}$. Then there exist $%
h_{0}>0$ and an open neighborhood $U$ of $\Gamma$ such that the order $%
\gamma $ LLRK discretization
\begin{equation*}
\mathbf{y}_{n+1}=\mathbf{y}_{n}+h\mathbf{\varphi}_{\gamma}(\mathbf{y}_{n},h)
\end{equation*}
has an invariant closed curve $\Gamma_{h}\subset U$ for all $h\leq h_{0}$.
More precisely, there exist $T-$periodic functions $\overline{\mathbf{y}}%
_{h}:\mathbb{R}\rightarrow U$ and $\sigma_{h}-1:\mathbb{R}\rightarrow
\mathbb{R}$ for $h\leq h_{0}$, which are uniformly Lipschitz and satisfy
\begin{equation*}
\overline{\mathbf{y}}_{h}(t)+h\mathbf{\varphi}_{\gamma}(\overline{\mathbf{y}}%
_{h}(t),h)=\overline{\mathbf{y}}_{h}(\sigma_{h}(t)),\text{ }t\in\mathbb{R}
\end{equation*}
and
\begin{equation*}
\sigma_{h}(t)=t+h+O(h^{\gamma+1})\text{ uniformly for }t\in\mathbb{R}.
\end{equation*}
Furthermore, the curve $\Gamma_{h}=\{\overline{\mathbf{y}}_{h}(t):t\in
\lbrack0,T]\}$ converges to $\Gamma$ in the Lipschitz norm. In particular,
\begin{equation*}
\underset{t\in\mathbb{R}}{max}\left\Vert \overline{\mathbf{x}}(t)-\overline {%
\mathbf{y}}_{h}(t)\right\Vert =O(h^{\gamma})
\end{equation*}
and
\begin{equation*}
\underset{t_{1}\neq t_{2}}{sup}\frac{\left\Vert (\overline{\mathbf{x}}-%
\overline{\mathbf{y}}_{h})(t_{1})-(\overline{\mathbf{x}}-\overline {\mathbf{y%
}}_{h})(t_{2})\right\Vert }{\left\vert t_{1}-t_{2}\right\vert }\rightarrow0%
\text{ as }h\rightarrow0.
\end{equation*}
\end{theorem}

\begin{proof}
Since Lemma \ref{Lemma 5.4} holds on a neighborhood of $\overline{\Omega }$,
it also holds on $\Omega $. In addition, Lemmas \ref{Lemma de phi LLT} and %
\ref{Lemma for Local Error LLRK} imply that $\mathbf{\varphi }_{\gamma }\in
\mathcal{C}^{2}(\overline{\Omega }\times \lbrack 0,h_{0}])$, so $\partial
\mathbf{\varphi }_{\gamma }\mathbf{/\partial y}_{n}$ is Lipschitz on $\Omega
$ uniformly in $h$. Thus, the hypotheses of Theorem 2.1 in \cite{Beyn 1987b}
hold for the LLRK discretizations of order $\gamma >2$, which completes the
proof.
\end{proof}

Theorem \ref{Prop. 5.5} affirms that, for $h$ sufficiently small, the LLRK
discretizations have a closed invariant curve $\Gamma_{h}$ , i.e., $(1+h%
\mathbf{\varphi}(.;h))(\Gamma_{h})=\Gamma_{h}$ , which converges to the
periodic orbit $\Gamma$ of the continuous system.

The next theorem deals with the behavior of the discrete trajectories of
LLRK discretizations near the invariant curve $\Gamma_{h}$ when the ODE (\ref%
{ODE-SS-0a}) has a stable periodic orbit $\Gamma$. For $\mathbf{x}_{0}$ in a
neighborhood of $\Gamma$, the notations
\begin{equation*}
W_{h}(\mathbf{x}_{0})=\{\mathbf{y}_{n}(\mathbf{x}_{0}):n\geq0\}\text{ \ and
\ \ }w(\mathbf{x}_{0})=\{\mathbf{x}(t;\mathbf{x}_{0}):t\geq0\}
\end{equation*}
will be used. In addition,
\begin{equation*}
d(A,B)=max\{\underset{\mathbf{z}\in A}{\sup}\text{ dist}(\mathbf{z},B),%
\underset{\mathbf{z}\in B}{\sup}\text{ dist}(\mathbf{z},A)\}
\end{equation*}
will denote the Hausdorff distance between two sets $A$ and $B$.\newline

\begin{theorem}
\label{Prop. 5.7}Let $\Gamma $ be a stable closed orbit of the equation (\ref%
{ODE-SS-0a}). Then, under the assumptions of Theorem \ref{Prop. 5.5}, there
exist $h_{0}$, $\alpha $, $\beta $, $C$ and $\rho >0$ such that for $h\leq
h_{0}$ and $dist(\mathbf{x}_{0},\Gamma _{h})\leq \rho $ the following holds:
\begin{equation*}
dist(\mathbf{y}_{n}(\mathbf{x}_{0}),\Gamma _{h})\leq C\text{ }\exp (-\alpha
t_{n})\text{ }dist(\mathbf{x}_{0},\Gamma _{h})
\end{equation*}%
and%
\begin{equation*}
dist(\mathbf{y}_{n}(\mathbf{x}_{0}),w(\mathbf{x}_{0}))\leq C(h^{\gamma
}+min\{h^{\gamma }\exp (\beta t_{n}),\exp (-\alpha t_{n})\})
\end{equation*}%
for $n\geq 0$. Moreover, for any $\delta >0$ there exist $\rho (\delta )$, $%
h(\delta )>0$ such that
\begin{equation*}
\underset{n\geq 0}{sup}\left\{ dist(\mathbf{y}_{n}(\mathbf{x}_{0}),w(\mathbf{%
x}_{0}))\right\} \leq Ch^{\gamma -\delta }
\end{equation*}%
for $h\leq h(\delta )$ and $dist(\mathbf{x}_{0},\Gamma _{h})\leq \rho
(\delta )$. Finally,
\begin{equation*}
d(W_{h}(\mathbf{x}_{0}),w(\mathbf{x}_{0}))\rightarrow 0\text{ as }%
h\rightarrow 0
\end{equation*}%
uniformly for $dist(\mathbf{x}_{0},\Gamma )\leq \rho $.
\end{theorem}

\begin{proof}
It can be proved in a similar way as Theorem \ref{Prop. 5.5}, but using
Theorem 3.2 in \cite{Beyn 1987b} instead of Theorem 2.1.
\end{proof}

This theorem states the stability of the invariant curve $\Gamma_{h}$ and
the convergence of the trajectories of a LLRK discretization to the
continuous trajectories of the underlying ODE when the discretization starts
at a point close enough to the stable periodic orbit $\Gamma$.

\section{A-stable explicit LLRK schemes\label{ODE LL Schemes}}

This section deals with practical issues of the LLRK methods, that is, with
the so called Local Linearization - Runge Kutta schemes.

Roughly speaking, every numerical implementation of a LLRK discretization
will be called LLRK scheme. More precisely, they are defined as follows.

\begin{definition}
\label{definition LLS} For an order $\gamma$ LLRK discretization
\begin{equation}
\mathbf{y}_{n+1}=\mathbf{y}_{n}+h_{n}\mathbf{\varphi}_{\gamma}(t_{n},\mathbf{%
y}_{n};h_{n}),  \label{ODE-LLS-22}
\end{equation}
as defined in (\ref{LLRK_Discretizat}), any recursion of the form
\begin{equation*}
\widetilde{\mathbf{y}}_{n+1}=\widetilde{\mathbf{y}}_{n}+h_{n}\widetilde {%
\mathbf{\varphi}}_{\gamma}\left( t_{n},\widetilde{\mathbf{y}}%
_{n};h_{n}\right) \mathbf{,}\text{ \ \ \ \ \ \ \ \ \ with }\widetilde{%
\mathbf{y}}_{0}=\mathbf{y}_{0},
\end{equation*}
where $\widetilde{\mathbf{\varphi}}_{\gamma}$ denotes some numerical
algorithm to compute $\mathbf{\varphi}_{\gamma}$, is called an LLRK scheme.
\end{definition}

When implementing the LLRK discretization (\ref{ODE-LLS-22}), that is, when
a LLRK scheme is constructed, the required evaluations of the expression $%
\mathbf{y}_{n}+\mathbf{\phi}\left( t_{n},\mathbf{y}_{n};.\right) $ at $%
t_{n+1}-t_{n}$ and $c_{i}\left( t_{n+1}-t_{n}\right) $ may be computed by
different algorithms. In \cite{de la Cruz 07}, \cite{Jimenez05 AMC} a number
of them were reviewed, which yield the following two basic kinds of LLRK
schemes:

\begin{equation*}
\widetilde{\mathbf{y}}_{n+1}=\widetilde{\mathbf{y}}_{n}+\widetilde{\mathbf{%
\phi }}\left( t_{n},\widetilde{\mathbf{y}}_{n};h_{n}\right) +\widetilde{%
\mathbf{\rho }}\left( t_{n},\widetilde{\mathbf{y}}_{n};h_{n}\right) \mathbf{,%
}
\end{equation*}%
and%
\begin{equation*}
\widetilde{\mathbf{y}}_{n+1}=\widetilde{\mathbf{z}}\left( t_{n}+h_{n};t_{n},%
\widetilde{\mathbf{y}}_{n}\right) +\widetilde{\mathbf{\rho }}\left( t_{n},%
\widetilde{\mathbf{y}}_{n};h_{n}\right) \mathbf{,}
\end{equation*}%
where $\widetilde{\mathbf{\phi }}$ is a\ numerical implementation of $%
\mathbf{\phi }$,\ $\widetilde{\mathbf{z}}$ is a numerical solution of the
linear ODE
\begin{align}
\frac{d\mathbf{z}\left( t\right) }{dt}& =\mathbf{B}_{n}\mathbf{z}(t)+\mathbf{%
b}_{n}\left( t\right) \mathbf{,}\text{ \ \ }t\in \lbrack t_{n},t_{n+1}],
\label{ODE-LLS-13} \\
\mathbf{z}\left( t_{n}\right) & =\widetilde{\mathbf{y}}_{n}\text{,}
\label{ODE-LLS-13b}
\end{align}%
and $\widetilde{\mathbf{\rho }}~$is the map of the Runge-Kutta scheme
applied to the ODE
\begin{align}
\frac{d\mathbf{v}\left( t\right) }{dt}& =\widetilde{\mathbf{q}}\mathbf{(}%
t_{n},\mathbf{z}\left( t_{n}\right) ;t\mathbf{,\mathbf{v}}\left( t\right)
\mathbf{),}\text{ \ \ }t\in \lbrack t_{n},t_{n+1}],\quad  \label{ODE-LLS-14}
\\
\mathbf{v}\left( t_{n}\right) & =\mathbf{0},  \label{ODE-LLS-14b}
\end{align}%
with vector field
\begin{align*}
\widetilde{\mathbf{q}}\mathbf{(}t_{n},\widetilde{\mathbf{y}}_{n};s\mathbf{%
,\xi )}& =\mathbf{f(}s,\widetilde{\mathbf{y}}_{n}+\widetilde{\mathbf{\phi }}%
\left( t_{n},\widetilde{\mathbf{y}}_{n};s-t_{n}\right) +\mathbf{\xi })-%
\mathbf{f}_{\mathbf{x}}(t_{n},\widetilde{\mathbf{y}}_{n})\widetilde{\mathbf{%
\phi }}\left( t_{n},\widetilde{\mathbf{y}}_{n};s-t_{n}\right) \\
& -\mathbf{f}_{t}\left( t_{n},\widetilde{\mathbf{y}}_{n}\right) (s-t_{n})-%
\mathbf{f}\left( t_{n},\widetilde{\mathbf{y}}_{n}\right) ,
\end{align*}%
for the first kind of LLRK scheme, or%
\begin{align*}
\widetilde{\mathbf{q}}\mathbf{(}t_{n},\widetilde{\mathbf{y}}_{n};s\mathbf{%
,\xi )}& =\mathbf{f(}s,\widetilde{\mathbf{z}}\left( s;t_{n},\widetilde{%
\mathbf{y}}_{n}\right) +\mathbf{\xi })-\mathbf{f}_{\mathbf{x}}(t_{n},%
\widetilde{\mathbf{y}}_{n})(\widetilde{\mathbf{z}}\left( s;t_{n},\widetilde{%
\mathbf{y}}_{n}\right) -\widetilde{\mathbf{y}}_{n})-\mathbf{f}_{t}\left(
t_{n},\widetilde{\mathbf{y}}_{n}\right) (s-t_{n}) \\
& -\mathbf{f}\left( t_{n},\widetilde{\mathbf{y}}_{n}\right)
\end{align*}%
for the second one. In the equation (\ref{ODE-LLS-13}), $\mathbf{B}_{n}=%
\mathbf{f}_{\mathbf{x}}\left( t_{n},\widetilde{\mathbf{y}}_{n}\right) $ is a
$d\times d$ constant matrix and $\mathbf{b}_{n}(t)=\mathbf{f}_{t}\left(
t_{n},\widetilde{\mathbf{y}}_{n}\right) (t-t_{n})+\mathbf{f}\left( t_{n},%
\widetilde{\mathbf{y}}_{n}\right) \mathbf{-B}_{n}\widetilde{\mathbf{y}}_{n}$
is a $d$-dimensional linear vector function.

Obviously, a LLRK scheme will preserve the order $\gamma$ of the underlaying
LLRK discretization only if $\widetilde{\mathbf{\phi}}$ is a suitable
approximation to $\mathbf{\phi}$. This requirement is considered in the next
theorem.

\begin{theorem}
\label{Conv LLRKScheme}Let $\mathbf{x}$ be the solution of the ODE (\ref%
{ODE-LLA-1})-(\ref{ODE-LLA-2}) with vector field $\mathbf{f}$ satisfying the
condition (\ref{ODE-CONV-8}). With $t_{n},$ $t_{n+1}\in\left( t\right) _{h}$%
, let $\widetilde{\mathbf{z}}_{n+1}=\widetilde{\mathbf{z}}_{n}+h_{n}{\Lambda}%
_{1}\left( t_{n},\widetilde{\mathbf{z}}_{n};h_{n}\right) $ and $\widetilde{%
\mathbf{v}}_{n+1}=\widetilde{\mathbf{v}}_{n}+h_{n}{\Lambda }_{2}^{\widetilde{%
\mathbf{z}}_{n}}\left( t_{n},\widetilde{\mathbf{v}}_{n};h_{n}\right) $ be
one-step explicit integrators of the ODEs (\ref{ODE-LLS-13})-(\ref%
{ODE-LLS-13b}) and (\ref{ODE-LLS-14})-(\ref{ODE-LLS-14b}), respectively.
Suppose that these integrators have order of convergence $r$ and $p$,
respectively. Further, assume that ${\Lambda}_{1}$ and ${\Lambda}_{2}^{%
\widetilde{\mathbf{z}}_{n}}$ fulfill the local Lipschitz condition (\ref%
{Lipschitz}). Then, for $h$ small enough, the numerical scheme%
\begin{equation*}
\widetilde{\mathbf{y}}_{n+1}=\widetilde{\mathbf{y}}_{n}+h_{n}{\Lambda}%
_{1}\left( t_{n},\widetilde{\mathbf{y}}_{n};h_{n}\right) +h_{n}{\Lambda}%
_{2}^{\widetilde{\mathbf{y}}_{n}}\left( t_{n},\mathbf{0};h_{n}\right)
\end{equation*}
satisfies that
\begin{equation*}
\left\Vert \mathbf{x}(t_{n+1})-\widetilde{\mathbf{y}}_{n+1}\right\Vert \leq
Ch^{\min\{r,p\}}
\end{equation*}
for all $t_{n+1}\in\left( t\right) _{h}$, where $C$ is a positive constant.
\end{theorem}

\begin{proof}
Let $\mathcal{X}=\left\{ \mathbf{x}\left( t\right) :t\in\left[ t_{0},T\right]
\right\} .$ Since $\mathcal{X}$ is a compact set contained in the open set $%
\mathcal{D}\subset\mathbb{R}^{d}$, there exists $\varepsilon>0$ such that
the compact set
\begin{equation*}
\mathcal{A}_{\varepsilon}=\left\{ \xi\in\mathbb{R}^{d}:\underset{\mathbf{x}%
\left( t\right) \in\mathcal{X}}{\min}\left\Vert \xi -\mathbf{x}\left(
t\right) \right\Vert \leq\varepsilon\right\}
\end{equation*}
is contained in $\mathcal{D}$.

First, set $\widetilde{\mathbf{y}}_{n}=\mathbf{x}(t_{n})$ in the equations (%
\ref{ODE-LLS-13})-(\ref{ODE-LLS-13b}) and (\ref{ODE-LLS-14})-(\ref%
{ODE-LLS-14b}). Since $\mathbf{x}\left( t_{n+1}\right) =\mathbf{y}%
_{LL}\left( t_{n}+h_{n};t_{n},\mathbf{x}(t_{n})\right) +\mathbf{r}\left(
t_{n}+h_{n};t_{n},\mathbf{x}(t_{n})\right) $, it is obtained that
\begin{align}
\left\Vert
\begin{array}{c}
\mathbf{x}(t_{n+1})-\mathbf{x}(t_{n})-h_{n}{\Lambda }_{1}(t_{n},\mathbf{x}%
(t_{n});h_{n}) \\
-h_{n}{\Lambda }_{2}^{\mathbf{x}(t_{n})}(t_{n},\mathbf{0};h_{n})%
\end{array}%
\right\Vert & \leq \left\Vert \mathbf{\phi }\left( t_{n},\mathbf{x}%
(t_{n});h_{n}\right) -h_{n}{\Lambda }_{1}(t_{n},\mathbf{x}%
(t_{n});h_{n})\right\Vert  \notag \\
& +\left\Vert \mathbf{r}\left( t_{n}+h_{n};t_{n},\mathbf{x}(t_{n})\right) -%
\mathbf{v}(t_{n+1})\right\Vert  \notag \\
& +\left\Vert \mathbf{v}(t_{n+1})-h_{n}{\Lambda }_{2}^{\mathbf{x}%
(t_{n})}(t_{n},\mathbf{0};h_{n})\right\Vert ,  \label{ODE-LLS-15}
\end{align}%
where $\mathbf{v}(t_{n+1})$ is the solution of equation (\ref{ODE-LLS-14})-(%
\ref{ODE-LLS-14b}) at $t=t_{n+1}$.

By definition, $\mathbf{r}\left( t_{n}+h_{n};t_{n},\mathbf{x}(t_{n})\right) $
is solution of the differential equation%
\begin{align*}
\frac{d\mathbf{u}\left( t\right) }{dt}& =\mathbf{q(}t_{n},\mathbf{x}(t_{n});t%
\mathbf{,\mathbf{u}}\left( t\right) \mathbf{),}\text{ \ \ }t\in \lbrack
t_{n},t_{n+1}],\quad \\
\mathbf{u}\left( t_{n}\right) & =\mathbf{0},
\end{align*}%
evaluated at $t=t_{n+1}$. Thus, by applying the "fundamental lemma" (see,
e.g., Theorem 10.2 in \cite{Hairer-Wanner93}), it is obtained that
\begin{equation}
\left\Vert \mathbf{r}\left( t;t_{n},\mathbf{x}(t_{n})\right) -\mathbf{v}%
(t)\right\Vert \leq \frac{\epsilon }{P}(e^{P(t-t_{n})}-1)  \label{ODE-LLS-16}
\end{equation}%
for $t\in \lbrack t_{n},t_{n+1}]$, where
\begin{align*}
\epsilon & =\underset{t\in \lbrack t_{n},t_{n+1}]}{sup}\left\Vert \mathbf{q(}%
t_{n},\mathbf{x}(t_{n});t\mathbf{,\mathbf{u}}\left( t\right) \mathbf{)}-%
\widetilde{\mathbf{q}}\mathbf{(}t_{n},\mathbf{x}(t_{n});t\mathbf{,\mathbf{u}}%
\left( t\right) \mathbf{)}\right\Vert \\
& \leq M\left\Vert \mathbf{\phi }\left( t_{n},\mathbf{x}(t_{n});h_{n}\right)
-h_{n}{\Lambda }_{1}(t_{n},\mathbf{x}(t_{n});h_{n})\right\Vert ,
\end{align*}%
$M=\underset{t\in \lbrack t_{0},T],\xi \in \mathcal{A}_{\varepsilon }}{2\sup
}\left\Vert \mathbf{f}_{\mathbf{x}}(t,\xi )\right\Vert $, and $P$ is the
Lipschitz constant of the function $\mathbf{q(}t_{n},\mathbf{x}(t_{n});\cdot
\mathbf{)}$ (which exists by Lemma \ref{Lemma for Local Error LLRK}).

Furthermore,
\begin{equation}
\left\Vert \mathbf{\phi }\left( t_{n},\mathbf{x}(t_{n});h_{n}\right) -h_{n}{%
\Lambda }_{1}(t_{n},\mathbf{x}(t_{n});h_{n})\right\Vert =\left\Vert \mathbf{z%
}\left( t_{n+1}\right) -\mathbf{z}\left( t_{n}\right) -h_{n}{\Lambda }%
_{1}(t_{n},\mathbf{z}\left( t_{n}\right) ;h_{n})\right\Vert ,
\label{ODE-LLS-17}
\end{equation}%
since $\mathbf{z}\left( t_{n+1}\right) =\mathbf{x}(t_{n})+\mathbf{\phi }%
\left( t_{n},\mathbf{x}(t_{n});h_{n}\right) $ is the solution (\ref%
{ODE-LLS-13})-(\ref{ODE-LLS-13b}) with $\widetilde{\mathbf{y}}_{n}=\mathbf{x}%
(t_{n})$ at $t=t_{n+1}$. On the other hand,%
\begin{equation}
\left\Vert \mathbf{z}\left( t_{n+1}\right) -\mathbf{z}\left( t_{n}\right)
-h_{n}{\Lambda }_{1}(t_{n},\mathbf{z}\left( t_{n}\right) ;h_{n})\right\Vert
\leq c_{1}h^{r+1}  \label{ODE-LLS-18}
\end{equation}%
and
\begin{equation}
\left\Vert \mathbf{v}(t_{n+1})-h_{n}{\Lambda }_{2}^{\mathbf{x}(t_{n})}(t_{n},%
\mathbf{0};h_{n})\right\Vert \leq c_{2}h^{p+1}  \label{ODE-LLS-19}
\end{equation}%
hold, since $\widetilde{\mathbf{z}}_{n+1}=\widetilde{\mathbf{z}}_{n}+h_{n}{%
\Lambda }_{1}\left( t_{n},\widetilde{\mathbf{z}}_{n};h_{n}\right) $ and $%
\widetilde{\mathbf{v}}_{n+1}=\widetilde{\mathbf{v}}_{n}+h_{n}{\Lambda }_{2}^{%
\widetilde{\mathbf{z}}_{n}}\left( t_{n},\widetilde{\mathbf{v}}%
_{n};h_{n}\right) $ are order $r$ and $p$ integrators, respectively. Here, $%
c_{1}$ and $c_{2}$ are positive constants independent of $h$.

From the inequalities (\ref{ODE-LLS-15})-(\ref{ODE-LLS-19}), the one-step
integrator%
\begin{equation*}
\widetilde{\mathbf{y}}_{n+1}=\widetilde{\mathbf{y}}_{n}+h_{n}{\Lambda }%
_{1}\left( t_{n},\widetilde{\mathbf{y}}_{n};h_{n}\right) +h_{n}{\Lambda }%
_{2}^{\widetilde{\mathbf{y}}_{n}}\left( t_{n},\mathbf{0};h_{n}\right)
\end{equation*}%
has local truncation error
\begin{equation*}
\left\Vert \mathbf{x}(t_{n+1})-\mathbf{x}(t_{n})-h_{n}{\Lambda }_{1}(t_{n},%
\mathbf{x}(t_{n});h_{n})-h_{n}{\Lambda }_{2}^{\mathbf{x}(t_{n})}(t_{n},%
\mathbf{0};h_{n})\right\Vert \leq c\text{ }h^{\min \{r,p\}+1},
\end{equation*}%
where $c=c_{1}+c_{2}+c_{1}M(e^{P}-1)/P$ is a positive constant. In addition,
since ${\Lambda }_{1}+$ ${\Lambda }_{2}^{\mathbf{x}(t_{n})}$ with fixed $%
t_{n},h_{n}$ is a local Lipschitz function on $\mathcal{D}$, Lemma 2 in \cite%
{Perko01} (pp. 92) implies that ${\Lambda }_{1}+$ ${\Lambda }_{2}^{\mathbf{x}%
(t_{n})}$ is a Lipschitz function on $\mathcal{A}_{\varepsilon }\subset
\mathcal{D}$. Thus, the stated estimate $\left\Vert \mathbf{x}(t_{n+1})-%
\widetilde{\mathbf{y}}_{n+1}\right\Vert \leq Ch^{\min \{r,p\}}$ for the
global error of the LLRK scheme $\widetilde{\mathbf{y}}_{n+1}$
straightforwardly follows from Theorem 3.6 in \cite{Hairer-Wanner93}, where $%
C$ is a positive contant. Finally, in order to guarantee that $\mathbf{y}%
_{n+1}\in \mathcal{A}_{\varepsilon }$ for all $n=0,...,N-1,$ and so that the
LLRK scheme is well-defined, it is sufficient that $0<h<\delta $, where $%
\delta $ is chosen in such a way that $C\delta ^{\min \{r,p\}}\leq
\varepsilon $.
\end{proof}

As an example, consider the computation of the function $\mathbf{\phi}$
through a Pad\'{e} approximation combined with the\ "scaling and squaring"
strategy for exponential matrices \cite{Golub 1989}. To do so, note that $%
\mathbf{\phi}$ can be written as \cite{Jimenez02 AML}, \cite{Jimenez05 AMC}%
\begin{equation*}
\mathbf{\phi}\left( t_{n},\widetilde{\mathbf{y}}_{n};h_{n}\right) =\mathbf{L}%
e^{\widetilde{\mathbf{D}}_{n}h_{n}}\mathbf{r,}
\end{equation*}
where
\begin{equation*}
\widetilde{\mathbf{D}}_{n}=\left[
\begin{array}{ccc}
\mathbf{f}_{\mathbf{x}}(t_{n},\widetilde{\mathbf{y}}_{n}) & \mathbf{f}%
_{t}(t_{n},\widetilde{\mathbf{y}}_{n}) & \mathbf{f}(t_{n},\widetilde {%
\mathbf{y}}_{n}) \\
0 & 0 & 1 \\
0 & 0 & 0%
\end{array}
\right] \in\mathbb{R}^{(d+2)\times(d+2)},
\end{equation*}
$\mathbf{L}=\left[
\begin{array}{ll}
\mathbf{I}_{d} & \mathbf{0}_{d\times2}%
\end{array}
\right] $ and $\mathbf{r}^{\intercal}=\left[
\begin{array}{ll}
\mathbf{0}_{1\times(d+1)} & 1%
\end{array}
\right] $ in case of non-autonomous ODEs; and
\begin{equation*}
\widetilde{\mathbf{D}}_{n}=\left[
\begin{array}{cc}
\mathbf{f}_{\mathbf{x}}(\widetilde{\mathbf{y}}_{n}) & \mathbf{f}(\widetilde{%
\mathbf{y}}_{n}) \\
0 & 0%
\end{array}
\right] \in\mathbb{R}^{(d+1)\times(d+1)},
\end{equation*}
$\mathbf{L}=\left[
\begin{array}{ll}
\mathbf{I}_{d} & \mathbf{0}_{d\times1}%
\end{array}
\right] $ and $\mathbf{r}^{\intercal}=\left[
\begin{array}{ll}
\mathbf{0}_{1\times d} & 1%
\end{array}
\right] $ for autonomous equations.

\begin{proposition}
Set $\widetilde{\mathbf{\phi }}\left( t_{n},\widetilde{\mathbf{y}}%
_{n};h_{n}\right) =\mathbf{L}$ $(\mathbf{P}_{p,q}(2^{-\kappa _{n}}\widetilde{%
\mathbf{D}}_{n}h_{n}))^{2^{\kappa _{n}}}$ $\mathbf{r}$, where $\mathbf{P}%
_{p,q}(2^{-\kappa _{n}}\widetilde{\mathbf{D}}_{n}h_{n})$ is the $(p,q)$-Pad%
\'{e} approximation of $e^{2^{-\kappa _{n}}\widetilde{\mathbf{D}}_{n}h_{n}}$%
, $\kappa _{n}$ is the smallest integer number such that $\left\Vert
2^{-\kappa _{n}}\widetilde{\mathbf{D}}_{n}h_{n}\right\Vert \leq \frac{1}{2}$%
, and the matrices $\widetilde{\mathbf{D}}_{n}$,$\mathbf{L}$\textbf{,} $%
\mathbf{r}$ are defined as above. Further, let $\widetilde{\mathbf{\rho }}~$%
be the numerical solution of the ODE (\ref{ODE-LLS-14})-(\ref{ODE-LLS-14b})
given by an order $\gamma $ explicit Runge-Kutta scheme. Then, under the
assumptions of Theorem \ref{Local Error LLRK}, the global error of the LLRK
scheme%
\begin{equation}
\widetilde{\mathbf{y}}_{n+1}=\widetilde{\mathbf{y}}_{n}+\widetilde{\mathbf{%
\phi }}\left( t_{n},\widetilde{\mathbf{y}}_{n};h_{n}\right) +\widetilde{%
\mathbf{\rho }}\left( t_{n},\widetilde{\mathbf{y}}_{n};h_{n}\right)
\label{LLRK scheme}
\end{equation}%
for the integration of the ODE (\ref{ODE-LLA-1})-(\ref{ODE-LLA-2}) is given
by
\begin{equation*}
\left\Vert \mathbf{x}(t_{n})-\widetilde{\mathbf{y}}_{n}\right\Vert \leq
Mh^{\min \{\gamma ,p+q\}}
\end{equation*}%
for all $t_{n}\in \left( t\right) _{h}$, where $M$ is a positive constant$.$
\end{proposition}

\begin{proof}
Let $\mathcal{K}\subset \mathcal{D}$ be a compact set. Since $\mathbf{P}%
_{p,q}$ is an analytical function on the unit circle, it is also a Lipschitz
function on this region. This and condition $\left\Vert 2^{-\kappa _{n}}%
\widetilde{\mathbf{D}}_{n}h_{n}\right\Vert \leq \frac{1}{2}$ for all $%
t_{n}\in \left( t\right) _{h}$ imply that there exists a positive constant $%
L $ such that
\begin{equation*}
\left\Vert \widetilde{\mathbf{\phi }}\left( t_{n},\mathbf{\xi }%
_{2};h_{n}\right) -\widetilde{\mathbf{\phi }}\left( t_{n},\mathbf{\xi }%
_{1};h_{n}\right) \right\Vert \leq L\left\Vert \mathbf{\xi }_{2}-\mathbf{\xi
}_{1}\right\Vert
\end{equation*}%
for all $\mathbf{\xi }_{1},\mathbf{\xi }_{2}\in \mathcal{K}$ and $t_{n}\in
\left( t\right) _{h}$. On the other hand, Lemma 4.1 in \cite{Jimenez12 BIT}
implies that there exists a positive constant $M$ such that
\begin{equation*}
\left\Vert \mathbf{z}(t_{n+1})-\mathbf{z}(t_{n})-\widetilde{\mathbf{\phi }}%
\left( t_{n},\mathbf{z}(t_{n});h_{n}\right) \right\Vert \leq Mh^{p+q+1}
\end{equation*}%
for all $t_{n}\in \left( t\right) _{h}$, where $\mathbf{z}$ is the solution
of the linear ODE (\ref{ODE-LLS-13})-(\ref{ODE-LLS-13b}).

In addition, since $\widetilde{\mathbf{\rho}}$ is an order\ $\gamma$
approximation to the solution of (\ref{ODE-LLS-14})-(\ref{ODE-LLS-14b}) that
satisfies the condition (\ref{Lipschitz}), the hypotheses of Theorem \ref%
{Conv LLRKScheme} hold, which completes the proof.
\end{proof}

The next theorem presents a way to define a class of A-stable LLRK schemes
on the basis of Pad\'{e} approximations to matrix exponentials.

\begin{theorem}
LLRK schemes of the form (\ref{LLRK scheme}) are A-stable if the $(p,q)$-Pad%
\'{e} approximation is taken with $p\leq q\leq p+2$. Moreover, if $q=p+1$ or
$q=p+2$, then such LLRK schemes are also L-stable.
\end{theorem}

\begin{proof}
Consider the scalar test equation%
\begin{equation*}
dx\left( t\right) =\lambda x\left( t\right) dt,
\end{equation*}
where $\lambda$ is a complex number with non-positive real part.

An LLRK scheme of the form (\ref{LLRK scheme}) applied to this autonomous
equation results in the recurrence%
\begin{align}
\widetilde{y}_{n+1}& =\widetilde{y}_{n}+\widetilde{\mathbf{\phi }}\left(
t_{n},\widetilde{y}_{n};h_{n}\right)  \notag \\
& =\widetilde{y}_{n}+\mathbf{L}(\mathbf{P}_{p,q}(\mathbf{M}))^{2^{\kappa
_{n}}}\mathbf{r,}  \label{LLRK scheme for linear ODE}
\end{align}%
where $\mathbf{M}=2^{-\kappa _{n}}\widetilde{\mathbf{D}}_{n}h_{n}$ and
\begin{equation*}
\widetilde{\mathbf{D}}_{n}=\left[
\begin{array}{cc}
\lambda & \lambda \widetilde{y}_{n} \\
0 & 0%
\end{array}%
\right] .
\end{equation*}%
Here,
\begin{equation*}
\mathbf{P}_{p,q}(z)=\frac{\mathbf{N}_{p,q}(z)}{\mathbf{D}_{p,q}(z)}
\end{equation*}%
denotes the $(p,q)-$Pad\'{e} approximation to $e^{z}$, where
\begin{equation*}
\mathbf{N}_{p,q}(z)=1+\frac{p}{q+p}z+\frac{p(p-1)}{(q+p)(q+p-1)}\frac{z^{2}}{%
2!}+\ldots +\frac{p(p-1)...1}{(q+p)...(q+1)}\frac{z^{p}}{p!},
\end{equation*}%
and $\mathbf{D}_{p,q}(z)=\mathbf{N}_{q,p}(-z)$.

Since
\begin{equation*}
\left( \mathbf{M}\right) ^{j}=\left[
\begin{array}{cc}
\left( 2^{-\kappa _{n}}h_{n}\lambda \right) ^{j} & \text{ \ }\left(
2^{-\kappa _{n}}h_{n}\lambda \right) ^{j}\widetilde{y}_{n} \\
0 & 0%
\end{array}%
\right] ,
\end{equation*}%
it can be shown that
\begin{equation*}
\mathbf{N}_{p,q}(\mathbf{M})=\left[
\begin{array}{cc}
\mathbf{N}_{p,q}\left( 2^{-\kappa _{n}}h_{n}\lambda \right) & \text{ \ }%
\left( \mathbf{N}_{p,q}\left( 2^{-\kappa _{n}}h_{n}\lambda \right) -1\right)
\widetilde{y}_{n} \\
0 & 1%
\end{array}%
\right] .
\end{equation*}%
Likewise,%
\begin{equation*}
\mathbf{D}_{p,q}(\mathbf{M})=\left[
\begin{array}{cc}
\mathbf{D}_{p,q}\left( 2^{-\kappa _{n}}h_{n}\lambda \right) & \text{ \ }%
\left( \mathbf{D}_{p,q}\left( 2^{-\kappa _{n}}h_{n}\lambda \right) -1\right)
\widetilde{y}_{n} \\
0 & 1%
\end{array}%
\right] .
\end{equation*}%
Hence,
\begin{equation*}
\mathbf{D}_{p,q}^{-1}(\mathbf{M})=\left[
\begin{array}{cc}
\left( \mathbf{D}_{p,q}\left( 2^{-\kappa _{n}}h_{n}\lambda \right) \right)
^{-1} & \text{ \ }-\frac{\left( \mathbf{D}_{p,q}\left( 2^{-\kappa
_{n}}h_{n}\lambda \right) -1\right) }{\left( \mathbf{D}_{p,q}\left(
2^{-\kappa _{n}}h_{n}\lambda \right) \right) }\widetilde{y}_{n} \\
0 & 1%
\end{array}%
\right] .
\end{equation*}%
Therefore,%
\begin{align*}
\mathbf{P}_{p,q}(\mathbf{M})& =\mathbf{N}_{p,q}(\mathbf{M})\mathbf{D}%
_{p,q}^{-1}(\mathbf{M}) \\
& =\left[
\begin{array}{cc}
\frac{\mathbf{N}_{p,q}\left( 2^{-\kappa _{n}}h_{n}\lambda \right) }{\mathbf{D%
}_{p,q}\left( 2^{-\kappa _{n}}h_{n}\lambda \right) } & \text{ \ }\left(
\frac{\mathbf{N}_{p,q}\left( 2^{-\kappa _{n}}h_{n}\lambda \right) }{\mathbf{D%
}_{p,q}\left( 2^{-\kappa _{n}}h_{n}\lambda \right) }-1\right) \widetilde{y}%
_{n} \\
0 & 1%
\end{array}%
\right] ,
\end{align*}%
and so%
\begin{equation*}
\left( \mathbf{P}_{p,q}(\mathbf{M})\right) ^{2^{\kappa _{n}}}=\left[
\begin{array}{cc}
\left( \frac{\mathbf{N}_{p,q}\left( 2^{-\kappa _{n}}h_{n}\lambda \right) }{%
\mathbf{D}_{p,q}\left( 2^{-\kappa _{n}}h_{n}\lambda \right) }\right)
^{2^{\kappa _{n}}} & \text{ \ }\left( \left( \frac{\mathbf{N}_{p,q}\left(
2^{-\kappa _{n}}h_{n}\lambda \right) }{\mathbf{D}_{p,q}\left( 2^{-\kappa
_{n}}h_{n}\lambda \right) }\right) ^{2^{\kappa _{n}}}-1\right) \widetilde{y}%
_{n} \\
0 & 1%
\end{array}%
\right] .
\end{equation*}%
By substituting the above expression in (\ref{LLRK scheme for linear ODE})
it is obtained that
\begin{equation*}
\widetilde{y}_{n+1}=R(\lambda )\widetilde{y}_{n},
\end{equation*}%
where%
\begin{equation*}
R(\lambda )=\left( \frac{\mathbf{N}_{p,q}\left( 2^{-\kappa _{n}}h_{n}\lambda
\right) }{\mathbf{D}_{p,q}\left( 2^{-\kappa _{n}}h_{n}\lambda \right) }%
\right) ^{2^{\kappa _{n}}}.
\end{equation*}%
Since\ $\Re (2^{-\kappa _{n}}h_{n}\lambda )\leq 0$, Theorem 353A, pp. 238 in
\cite{Butcher 2008} implies that $\left\vert R(\lambda )\right\vert \leq 1$
for $p\leq q\leq p+2$. That is, for these values of $p$ and $q$ the LLRK
scheme (\ref{LLRK scheme}) is A-stable. The proof concludes by noting that,
for $q=p+1$ or $q=p+2$, $R(z)=0$ when $z\rightarrow \infty $.
\end{proof}

From an implementation viewpoint, further simplifications for LLRK schemes
can be achieved in order to reduce the computational budget of the
algorithms. For instance, if all the Runge Kutta coefficients $c_{i}$ have a
minimum common multiple $\kappa $, then the LLRK scheme (\ref{LLRK scheme})
can be implemented in terms of a few powers of the same matrix exponential $%
e^{\kappa h_{n}\widetilde{\mathbf{D}}_{n}}$. To illustrate this, let us
consider the so called \textit{four order classical Runge-Kutta scheme}
(see, e.g., pp. 180 in \cite{Butcher 2008}) with coefficients $c=\left[
\begin{array}{cccc}
0 & \frac{1}{2} & \frac{1}{2} & 1%
\end{array}%
\right] $. This yields the following efficient order 4 LLRK scheme
\begin{equation}
\widetilde{\mathbf{y}}_{n+1}=\widetilde{\mathbf{y}}_{n}+\widetilde{\mathbf{%
\phi }}\left( t_{n},\widetilde{\mathbf{y}}_{n};h_{n}\right) +\widetilde{%
\mathbf{\rho }}\left( t_{n},\widetilde{\mathbf{y}}_{n};h_{n}\right) ,
\label{LLRK4 scheme}
\end{equation}%
where%
\begin{equation*}
\widetilde{\mathbf{\rho }}\left( t_{n},\widetilde{\mathbf{y}}%
_{n};h_{n}\right) =\frac{h_{n}}{6}(2\widetilde{\mathbf{k}}_{2}+2\widetilde{%
\mathbf{k}}_{3}+\widetilde{\mathbf{k}}_{4}),
\end{equation*}%
with\ \ \ \ \ \ \
\begin{align*}
\widetilde{\mathbf{k}}_{i}& =\mathbf{f}\left( t_{n}+c_{i}h_{n},\widetilde{%
\mathbf{y}}_{n}+\widetilde{\mathbf{\phi }}(t_{n},\widetilde{\mathbf{y}}%
_{n};c_{i}h_{n})+c_{i}h_{n}\widetilde{\mathbf{k}}_{i-1}\right) -\mathbf{f}%
\left( t_{n},\widetilde{\mathbf{y}}_{n}\right) \\
& -\mathbf{f}_{\mathbf{x}}\left( t_{n},\widetilde{\mathbf{y}}_{n}\right)
\widetilde{\mathbf{\phi }}\left( t_{n},\widetilde{\mathbf{y}}%
_{n};c_{i}h_{n}\right) \ -\mathbf{f}_{t}\left( t_{n},\widetilde{\mathbf{y}}%
_{n}\right) c_{i}h_{n},
\end{align*}%
$\widetilde{\mathbf{k}}_{1}\equiv \mathbf{0}$, $\widetilde{\mathbf{\phi }}%
(t_{n},\widetilde{\mathbf{y}}_{n};\frac{h_{n}}{2})=\mathbf{LAr}$, $%
\widetilde{\mathbf{\phi }}(t_{n},\widetilde{\mathbf{y}}_{n};h_{n})=\mathbf{LA%
}^{2}\mathbf{r}$, $\mathbf{A}=(\mathbf{P}_{p,q}(2^{-\kappa _{n}}\widetilde{%
\mathbf{D}}_{n}h_{n}))^{2^{\kappa _{n}}}$,
\begin{equation*}
\widetilde{\mathbf{D}}_{n}=\left[
\begin{array}{ccc}
\mathbf{f}_{\mathbf{x}}(t_{n},\widetilde{\mathbf{y}}_{n}) & \mathbf{f}%
_{t}(t_{n},\widetilde{\mathbf{y}}_{n}) & \mathbf{f}(t_{n},\widetilde{\mathbf{%
y}}_{n}) \\
0 & 0 & 1 \\
0 & 0 & 0%
\end{array}%
\right] \in \mathbb{R}^{(d+2)\times (d+2)},
\end{equation*}%
and $\kappa _{n}$ is the smallest integer number such that $\left\Vert
2^{-\kappa _{n}}\widetilde{\mathbf{D}}_{n}h_{n}\right\Vert \leq \frac{1}{2}$.

Note that the dynamical properties of an order $\gamma$ LLRK discretization,
as stated in section \ref{Sec. teoria}, are inherited by its numerical
implementations if the approximation to the map $\mathbf{\phi}+$ $\mathbf{%
\rho}$ is $o(h^{\gamma-1})$ and smooth enough (i.e., of class $C^{\gamma}$).
In particular, these conditions are satisfied by the implementations just
introduced, namely, those given by (\ref{LLRK scheme}). This provides
theoretical support to the simulation study presented in \cite{de la Cruz 06}%
, \cite{de la Cruz Ph.D. Thesis}, which reports satisfactory dynamical
behavior of LLRK schemes in the neighborhood of invariant sets of ODEs.

Finally note that, as an example, this section has focused on a specify kind
of LLRK scheme, namely, the A-stable scheme (\ref{LLRK4 scheme}) that
combines the A-stable Pad\'{e} algorithm to compute the $\mathbf{\varphi }%
_{\gamma }$ with the $4$ order classical Runge-Kutta scheme to compute the
solution of the auxiliary equation (\ref{ODE-LLS-14})-(\ref{ODE-LLS-14b}).
However, because of the flexibility in the numerical implementation of the
LLRK methods, specific schemes can be designed for certain classes of ODEs,
i.e., LLRK schemes based on L-stable Pad\'{e} algorithm and Rosenbrock
schemes for stiff equations; or LLRK schemes based on Krylov algorithm in
case of high dimensional ODEs, etc. For all of them the results of this
section also apply.

\section{Numerical simulations}

In this section, the performance of the LLRK$4$ scheme (\ref{LLRK4 scheme})
is illustrated by means of numerical simulations. To do so, a variety of
ODEs were selected. All simulations were carried out in Matlab2007b, and the
Matlab function\ "expm" was used in all computations involving exponential
matrices.

The first example is taken from \cite{Beyn 1987a} to illustrate the
dynamical behavior of the LLRK$4$ scheme in the neighborhood of hyperbolic
stationary points. For comparative purposes, the order $2$ Local
Linearization scheme of \cite{Jimenez02 AMC}, and a straightforward
non-adaptive implementation of the order $5$ Runge-Kutta formula of Dormand
\& Prince \cite{Dormand80}\ (used in Matlab2007b) are considered too. They
will be denoted by LL$2$ and RK$45$, respectively.

\textbf{Example 1}

\begin{align}
\frac{d\mathbf{x}_{1}}{dt} & =-2\mathbf{x}_{1}+\mathbf{x}_{2}+1-\mu f\left(
\mathbf{x}_{1},\lambda\right) ,  \label{EJ1-E1} \\
\frac{d\mathbf{x}_{2}}{dt} & =\mathbf{x}_{1}-2\mathbf{x}_{2}+1-\mu f\left(
\mathbf{x}_{2},\lambda\right) ,  \label{EJ1-E2}
\end{align}
where $f\left( u,\lambda\right) =u\left( 1+u+\lambda u^{2}\right) ^{-1}$.

For $\mu=15$, $\lambda=57$, this system has two stable stationary points and
one unstable stationary point in the region $0\leq x_{1},x_{2}\leq1$. There
is a nontrivial stable manifold for the unstable point which separates the
basins of attraction for the two stable points.

Figure 1a) presents the phase portrait obtained by the LLRK$4$ scheme with a
very small step-size $\left( h=2^{-13}\right) $, which can be regarded as
the exact solution for comparative purposes. The stable manifold $M_{s}$ of
the unstable point was found by bisection. Figures 1b), 1c) and 1d) show the
phase portraits obtained, respectively, by the LL$2$, the RK$45$ and the LLRK%
$4$ schemes with step-size $h=2^{-2}$ fixed. It can be observed that the RK$%
45$ discretization fails to reproduce correctly the phase portrait of the
underlying system near one of the point attractors. On the contrary, the
exact phase portrait is adequately approximated near both point attractors
by the LL$2$ and LLRK$4$ schemes, being the latter much more accurate. Other
significant difference in the integration of this equation appears near to
the stable manifold $M_{s}$. Changes in the intersection point $(0,\xi _{h})$
of the approximate stable manifold $M_{s}^{h}$ with the $x_{2}$-axis is
shown in Table I for the considered schemes. The values of $\xi _{h}$ were
calculated by a bisection method and the estimated order of convergence was
calculated as
\begin{equation*}
r_{h}=\frac{1}{\ln 2}\ln (\frac{\xi _{h}-\xi _{h/2}}{\xi _{h/2}-\xi _{h/4}}).
\end{equation*}%
For $h<2^{-4}$, the reported values of $r_{h}$ for the schemes LL$2$ and LLRK%
$4$ are in concordance with the expected asymptotic behavior $\xi _{h}=\xi
_{0}+Ch^{r}+O(h^{r+1})$ stated by Theorem \ref{Prop. 5.3} and Theorems 3 in
\cite{Jimenez02 AMC}, respectively, but not with the stated by Theorem 3.1
in \cite{Beyn 1987a} for the RK$45$, i.e., $r_{h}\approx 5$. This means that
the LL$2$ and LLRK$4$ schemes provide better approximations to the stable
and unstable manifolds on bigger neighborhoods of the equilibrium points,
which is obviously a favorable result for them. These results show out too
that the LLRK4 scheme preserves much better the basins of attraction of the
ODE (\ref{EJ1-E1})-(\ref{EJ1-E2}) than the RK$45$ and LL$2$ schemes.

\begin{figure}[h]
\begin{center}
\includegraphics[width=1.0\textwidth]{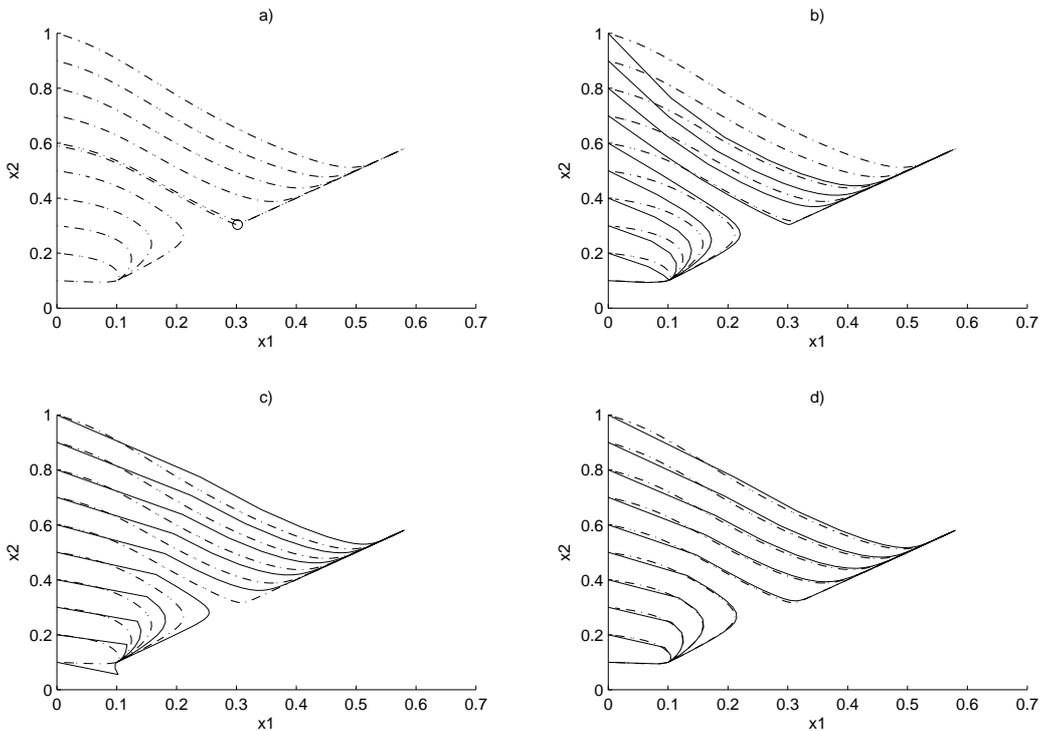}%
\caption{Phase portrait of the system (\ref{EJ1-E1}%
)-(\ref{EJ1-E2}) computed with fixed step-size $h$ by a) LLRK4 scheme with
$h=10^{-13}$ (dashed line). The unstable point is pointed out with "o"; b)
LL2 scheme, with $h=2^{-2}$; c) RK45 scheme, with $h=2^{-2}$; d) LLRK4 scheme,
with $h=2^{-2}$. In all cases the solid lines represent the solution computed
with $h=2^{-2}$.}%
\end{center}
\end{figure}

In what follows, we compare the accuracy of the LLRK$4$ scheme with those of
the LL$2$ scheme, and the Matlab2007b codes ode$45$ and ode$15s$ in the
integration of a variety of ODEs. We recall that the code ode$45$ is a
variable step-size implementation of the explicit Runge-Kutta $(4,5)$ pair
of Dormand \& Prince \cite{Dormand80}, which is considered for many authors
the most recommendable scheme to apply as a first try for most problems. On
the other hand, the code ode$15s$ is a quasi-constant step-size
implementation in terms of backward differences of the Klopfenstein-Shampine
family of numerical differentiation formulas of orders $1-5$, which is
designed for stiff problems when the ode$45$ fails to provide desired result
\cite{Shampine97}.\newline

\begin{center}
$\ $%
\begin{tabular}
[c]{|c|c|c|c|}\hline
Step-size%
$\backslash$%
Scheme & LL$2$ & RK$45$ & LLRK$4$\\\hline%
\begin{tabular}
[c]{l}%
$h$\\
$2^{-1}$\\
$2^{-2}$\\
$2^{-3}$\\
$2^{-4}$\\
$2^{-5}$\\
$2^{-6}$\\
$2^{-7}$\\
$2^{-8}$%
\end{tabular}
&
\begin{tabular}{ll}
$\xi_{h}$ & $r_{h}$ \\
\multicolumn{1}{c}{$0.71911$} & \multicolumn{1}{c}{} \\
\multicolumn{1}{c}{$0.69688$} & \multicolumn{1}{c}{$1.931$} \\
\multicolumn{1}{c}{$0.61727$} & \multicolumn{1}{c}{$2.190$} \\
\multicolumn{1}{c}{$0.59639$} & \multicolumn{1}{c}{$2.145$} \\
\multicolumn{1}{c}{$0.59182$} & \multicolumn{1}{c}{$2.056$} \\
\multicolumn{1}{c}{$0.59079$} & \multicolumn{1}{c}{$2.027$} \\
\multicolumn{1}{c}{$0.59054$} & $2.014$ \\
\multicolumn{1}{c}{$0.59048$} &
\end{tabular}
&
\begin{tabular}{ll}
$\xi_{h}$ & $r_{h}$ \\
$0.70377$ &  \\
$0.53673$ & $3.00$ \\
$0.59859$ & $4.18$ \\
$0.59088$ & $7.23$ \\
$0.590458$ & $6.93$ \\
$0.59045593$ & $6.39$ \\
$0.5904559168$ & $5.98$ \\
$0.5904559165$ &
\end{tabular}
&
\begin{tabular}{ll}
$\xi_{h}$ & $r_{h}$ \\
$0.56615$ &  \\
$0.58441$ & $5.142$ \\
$0.59032$ & $2.384$ \\
$0.59049$ & $3.354$ \\
$0.590459$ & $3.901$ \\
$0.5904561$ & $3.973$ \\
$0.590455917$ & $3.989$ \\
$0.590455916$ &
\end{tabular}
\\ \hline
\end{tabular}
\end{center}

{\small Table I. Values of }$\xi_{h}${\small \ and }$r_{h}${\small \
computed by the LL}${\small 2}${\small , RK}${\small 45}$ {\small and LLRK}$%
{\small 4}$ {\small schemes in the integration of the system (\ref{EJ1-E1})-(%
\ref{EJ1-E2}), for different values of }$h${\small .}\newline

In order to compare the (non-adaptive) LL schemes with the adaptive Matlab
codes, the following procedure was carried out. First, one of the Matlab
codes is used \ to compute the solution with fixed values of relative ($RT$)
and absolute ($AT$) tolerance. Then, the resulting integration steps $%
(t)_{h} $ are set as input in the other schemes for obtaining solutions at
the same integration steps. Second, the Matlab code ode15s is used to
compute on $(t)_{h}$ a very accurate solution $\mathbf{z}$ with $%
RT=RA=10^{-13}$. Third, the approximate solution $\mathbf{y}$ of the ODE is
computed for each scheme on $(t)_{h}$, and the relative error%
\begin{equation*}
RE=\underset{{\small i=1,\ldots ,d;}\text{ }{\small t}_{j}{\small \in (t)}%
_{h}}{\max }\left\vert \frac{\mathbf{z}_{i}(t_{j})-\mathbf{y}_{i}(t_{j})}{%
\mathbf{z}_{i}(t_{j})}\right\vert
\end{equation*}%
is evaluated.

The following four examples are of the form%
\begin{equation}
\frac{d\mathbf{x}}{dt}=\mathbf{Ax+f}(\mathbf{x),}  \label{Semi-Linear}
\end{equation}%
where $\mathbf{A}$ is a square constant matrix, and $\mathbf{f}$ is a
nonlinear function of $\mathbf{x}$. The vector field of the first two ones
has Jacobians with eigenvalues on or near to the imaginary axis, which make
these oscillators difficulty to be integrated by a number of conventional
integrators \cite{Gaffney84,Shampine97}. The other two are also hard for
conventional explicit schemes since they are examples of stiff equations
\cite{Shampine97}. Example 5 has an additional complexity for a number of
integrators that do not update the Jacobians of the vector field at each
integration step \cite{Shampine97,Hochbruck-etal09}: the Jacobian of the
linear term has positive eigenvalues, which results a problem for the
integration in a neighborhood of the stable equilibrium point $\mathbf{x}=1$.

\textbf{Example 2}. Periodic linear:%
\begin{equation*}
\frac{d\mathbf{x}}{dt}=\mathbf{A}(\mathbf{x}+2),
\end{equation*}%
with%
\begin{equation*}
\mathbf{A}=\left[
\begin{array}{cc}
i & 0 \\
0 & -i%
\end{array}%
\right] ,
\end{equation*}%
$\mathbf{x}_{1}(t_{0})=-2.5$, $\mathbf{x}_{2}(t_{0})=-1.5$, and $%
[t_{0},T]=[0,4\pi ]$.

\textbf{Example 3. }Periodic linear plus nonlinear part:%
\begin{equation*}
\frac{d\mathbf{x}}{dt}=\mathbf{A}(\mathbf{x}+2)+0.1\mathbf{x}^{2},
\end{equation*}%
where the matrix $\mathbf{A}$ is defined as in the previous example, $%
\mathbf{x}(t_{0})=1$, and $[t_{0},T]=[0,4\pi ]$.

\textbf{Example 4. }Stiff equation:%
\begin{equation*}
\frac{d\mathbf{x}}{dt}=-100\mathbf{H}(\mathbf{x+1}),
\end{equation*}%
where $\mathbf{H}$ is the 12-dimensional Hilbert matrix (with conditioned
number $1.69\times 10^{16}$), $\mathbf{x}_{i}(t_{0})=1$, and $%
[t_{0},T]=[0,1] $.

\textbf{Example 5. }Stiff linear plus nonlinear part:%
\begin{equation*}
\frac{d\mathbf{x}}{dt}=100\mathbf{H}(\mathbf{x}-\mathbf{1})+100(\mathbf{x}-%
\mathbf{1})^{2}-60(\mathbf{x}^{3}-\mathbf{1}),
\end{equation*}%
where $\mathbf{H}$ is the 12-dimensional Hilbert matrix, $\mathbf{x}%
_{i}(t_{0})=-0.5$, and $[t_{0},T]=[0,1]$.
\newline
\newline

\begin{center}
$\ $%
\begin{tabular}{|c|c|c|c|c|c|}
\hline
{\small Example} & {\small Scheme} & $%
\begin{array}{c}
\text{{\small Relative}} \\
\text{{\small Tolerance}}%
\end{array}
$ & $%
\begin{array}{c}
\text{{\small Absolute}} \\
\text{{\small Tolerance}}%
\end{array}
$ & $%
\begin{array}{c}
\text{{\small NS}} \\
\end{array}
$ & $%
\begin{array}{c}
\text{{\small Relative}} \\
\text{{\small Error}}%
\end{array}
$ \\ \hline
\multicolumn{1}{|l|}{$%
\begin{array}{c}
\text{{\small 2 : Periodic linear}}%
\end{array}
$} & $%
\begin{array}{c}
{\small ode15s}^{\ast} \\
{\small ode45} \\
{\small LL2} \\
{\small LLRK4}%
\end{array}
$ & $%
\begin{array}{c}
{\small 10}^{-3} \\
{\small 5\times10}^{-6} \\
{\small -} \\
{\small -}%
\end{array}
$ & $%
\begin{array}{c}
{\small 10}^{-6} \\
{\small 5\times10}^{-9} \\
{\small -} \\
{\small -}%
\end{array}
$ & $%
\begin{array}{c}
{\small 334} \\
{\small 340} \\
{\small 334} \\
{\small 334}%
\end{array}
$ & \multicolumn{1}{|l|}{$%
\begin{array}{c}
{\small 0.19} \\
{\small 8.2\times10}^{-5} \\
{\small 1.6\times10}^{-12} \\
{\small 1.6\times10}^{-12}%
\end{array}
$} \\ \hline
\multicolumn{1}{|l|}{$%
\begin{array}{c}
\text{{\small 3 : Periodic linear}} \\
\text{{\small plus nonlinear part}}%
\end{array}
$} & $%
\begin{array}{c}
{\small ode15s}^{\ast} \\
{\small ode45} \\
{\small LL2} \\
{\small LLRK4}%
\end{array}
$ & $%
\begin{array}{c}
{\small 10}^{-3} \\
{\small 6\times10}^{-6} \\
{\small -} \\
{\small -}%
\end{array}
$ & $%
\begin{array}{c}
{\small 10}^{-6} \\
{\small 6\times10}^{-9} \\
{\small -} \\
{\small -}%
\end{array}
$ & $%
\begin{array}{c}
{\small 287} \\
{\small 289} \\
{\small 287} \\
{\small 287}%
\end{array}
$ & \multicolumn{1}{|l|}{$%
\begin{array}{c}
{\small 0.30} \\
{\small 1.3\times10}^{-4} \\
{\small 3.1\times10}^{-2} \\
{\small 1.1\times10}^{-5}%
\end{array}
$} \\ \hline
\multicolumn{1}{|l|}{$%
\begin{array}{c}
\text{{\small 4 : Stiff linear}}%
\end{array}
$} & $%
\begin{array}{c}
{\small ode15s}^{\ast} \\
{\small ode45} \\
{\small LL2} \\
{\small LLRK4}%
\end{array}
$ & $%
\begin{array}{c}
{\small 10}^{-3} \\
{\small 5\times10}^{-4} \\
{\small -} \\
{\small -}%
\end{array}
$ & $%
\begin{array}{c}
{\small 10}^{-6} \\
{\small 5\times10}^{-7} \\
{\small -} \\
{\small -}%
\end{array}
$ & $%
\begin{array}{c}
{\small 66} \\
{\small 66} \\
{\small 66} \\
{\small 66}%
\end{array}
$ & \multicolumn{1}{|l|}{$%
\begin{array}{c}
{\small 6.7\times10}^{-2} \\
{\small 5.3\times10}^{-3} \\
{\small 1.8\times10}^{-10} \\
{\small 1.8\times10}^{-10}%
\end{array}
$} \\ \hline
\multicolumn{1}{|l|}{$%
\begin{array}{c}
\text{{\small 5 : Stiff linear \ \ \ \ \ \ \ }} \\
\text{{\small plus nonlinear part}}%
\end{array}
$} & $%
\begin{array}{c}
{\small ode15s}^{\ast} \\
{\small ode45} \\
{\small LL2} \\
{\small LLRK4}%
\end{array}
$ & $%
\begin{array}{c}
{\small 10}^{-2} \\
{\small 10}^{-1} \\
{\small -} \\
{\small -}%
\end{array}
$ & $%
\begin{array}{c}
{\small 10}^{-4} \\
{\small 10}^{-3} \\
{\small -} \\
{\small -}%
\end{array}
$ & $%
\begin{array}{c}
{\small 49} \\
{\small 104} \\
{\small 49} \\
{\small 49}%
\end{array}
$ & \multicolumn{1}{|l|}{$%
\begin{array}{c}
{\small 0.31} \\
{\small 0.37} \\
{\small 0.43} \\
{\small 4.3\times10}^{-5}%
\end{array}
$} \\ \hline
\multicolumn{1}{|l|}{$%
\begin{array}{c}
\text{{\small 6 : Nonlinear}} \\
\text{{\small (no stiff)}}%
\end{array}
$} & $%
\begin{array}{c}
{\small ode15s} \\
{\small ode45}^{\ast} \\
{\small LL2} \\
{\small LLRK4}%
\end{array}
$ & $%
\begin{array}{c}
{\small 10}^{-2} \\
{\small 10}^{-3} \\
{\small -} \\
{\small -}%
\end{array}
$ & $%
\begin{array}{c}
{\small 10}^{-5} \\
{\small 10}^{-6} \\
{\small -} \\
{\small -}%
\end{array}
$ & $%
\begin{array}{c}
{\small 103} \\
{\small 47} \\
{\small 47} \\
{\small 47}%
\end{array}
$ & $%
\begin{array}{c}
{\small 0.35} \\
{\small 0.08} \\
{\small 4.19} \\
{\small 0.25}%
\end{array}
$ \\ \hline
\multicolumn{1}{|l|}{$%
\begin{array}{c}
\text{{\small 7 : Nonlinear \ \ \ }} \\
\text{{\small (moderate stiff)}}%
\end{array}
$} & $%
\begin{array}{c}
{\small ode15s} \\
{\small ode45}^{\ast} \\
{\small LL2} \\
{\small LLRK4}%
\end{array}
$ & $%
\begin{array}{c}
{\small 1.5\times10}^{-9} \\
{\small 10}^{-7} \\
{\small -} \\
{\small -}%
\end{array}
$ & $%
\begin{array}{c}
{\small 1.5\times10}^{-12} \\
{\small 10}^{-10} \\
{\small -} \\
{\small -}%
\end{array}
$ & $%
\begin{array}{c}
{\small 2281} \\
{\small 2285} \\
{\small 2285} \\
{\small 2285}%
\end{array}
$ & \multicolumn{1}{|l|}{$%
\begin{array}{c}
{\small 1.2\times10}^{-3} \\
{\small 1.6\times10}^{-3} \\
{\small 400} \\
{\small 6.9\times10}^{-3}%
\end{array}
$} \\ \hline
\end{tabular}
\end{center}

{\small Table II. Accuracy of the LL2, LLRK4, ode45 and ode15s
schemes in the integration of examples (}${\small 2}${\small
)-(}${\small 7}${\small ).
With the symbol * is denoted the Matlab code used to set the time partition }%
${\small (t)}_{h}${\small \ in each example. NS denotes the number
of steps
required for each scheme to compute the solution on }${\small (t)}_{h}$%
{\small .}\newline

The results of the integration of these equations for each scheme
are shown in Table II. For illustration, Figure 2 shows the path of the variable $%
\mathbf{x}_{1}$ and its approximation $\mathbf{y}_{1}$ obtained by the LLRK$%
4 $ scheme in the integration of these equations. Remarkable, in all
the examples, the relative error of the solution obtained by the
LLRK$4$ scheme is much lower that those of the LL$2$, ode$45$ and
ode$15s$ with the same or lower number of steps. These results are
easily comprehensible for five reasons: 1) the dynamics of these
equations strongly depend on the linear part of their vector fields;
2) the LL$2$ and LLRK$4$ schemes preserve the stability of the
linear systems for all step-sizes, which is not so for conventional
explicit integrators; 3) the LL$2$ and LLRK$4$ schemes are able to\
"exactly" (up to the precision of the floating-point arithmetic)
integrate linear ODEs, which is a property not satisfied by for
conventional explicit and implicit schemes; 4) the LL$2$ and LLRK$4$
schemes update the exact Jacobian of the vector field at each
integration step, which is not done by most of conventional schemes;
and 5) the LLRK$4$ has higher order of convergence than the LL$2$
scheme. Further, note that although the LLRK$4$ scheme is not
designed for the integration of stiff ODEs in general (because the
auxiliary equation (\ref{ODE-LLS-14})-(\ref{ODE-LLS-14b}) might
\textquotedblleft inherit\textquotedblright\ the stiffness of the
original one) it is clear that, by construction, it is suitable for
equations with stiffness confined to the linear part. Example are
the classes of stiff linear and semilinear equations represented in
the Examples $2$ and $3$. This is so, because at each integration
step the stiff linear term is
locally removed from the vector field of the auxiliary equation (\ref%
{ODE-LLS-14}) and, in this way, the stiff linear part is well
integrated by
the (A-stable) LL scheme and the resulting non-stiff equation (\ref%
{ODE-LLS-14}) can be well integrated by the explicit RK scheme.

The following two examples are well known nonlinear oscillators.

\textbf{Example 6}. Non-stiff nonlinear:%
\begin{align*}
\frac{d\mathbf{x}_{1}}{dt}& =1+\mathbf{x}_{1}^{2}\mathbf{x}_{2}-4\mathbf{x}%
_{1}, \\
\frac{d\mathbf{x}_{2}}{dt}& =3\mathbf{x}_{1}-\mathbf{x}_{1}^{2}\mathbf{x}_{2}
\end{align*}%
where $\mathbf{x}_{1}(t_{0})=1.5$, $\mathbf{x}_{2}(t_{0})=3$, and $%
[t_{0},T]=[0,20]$. This equation, known as Brusselator equation, is a
typical test equation of non-stiff nonlinear problems (see, e.g., \cite%
{Hairer-Wanner93} ) .

\textbf{Example 7}. Mild-stiff nonlinear:%
\begin{align*}
\frac{d\mathbf{x}_{1}}{dt}& =\mathbf{x}_{2}, \\
\frac{d\mathbf{x}_{2}}{dt}& =\varepsilon ((1-\mathbf{x}_{2}^{2})\mathbf{x}%
_{1}+\mathbf{x}_{2}),
\end{align*}%
where $\varepsilon =10^{3}$, $\mathbf{x}_{1}(t_{0})=2$, $\mathbf{x}%
_{2}(t_{0})=0$, and $[t_{0},T]=[0,2]$. This equation, known as Van der Pol
equation, is a typical test equation of stiff nonlinear problems (see, e.g.,
\cite{Hairer-Wanner96} ).

\begin{figure}[ptb]
\begin{center}
\includegraphics[width=1.0\textwidth]{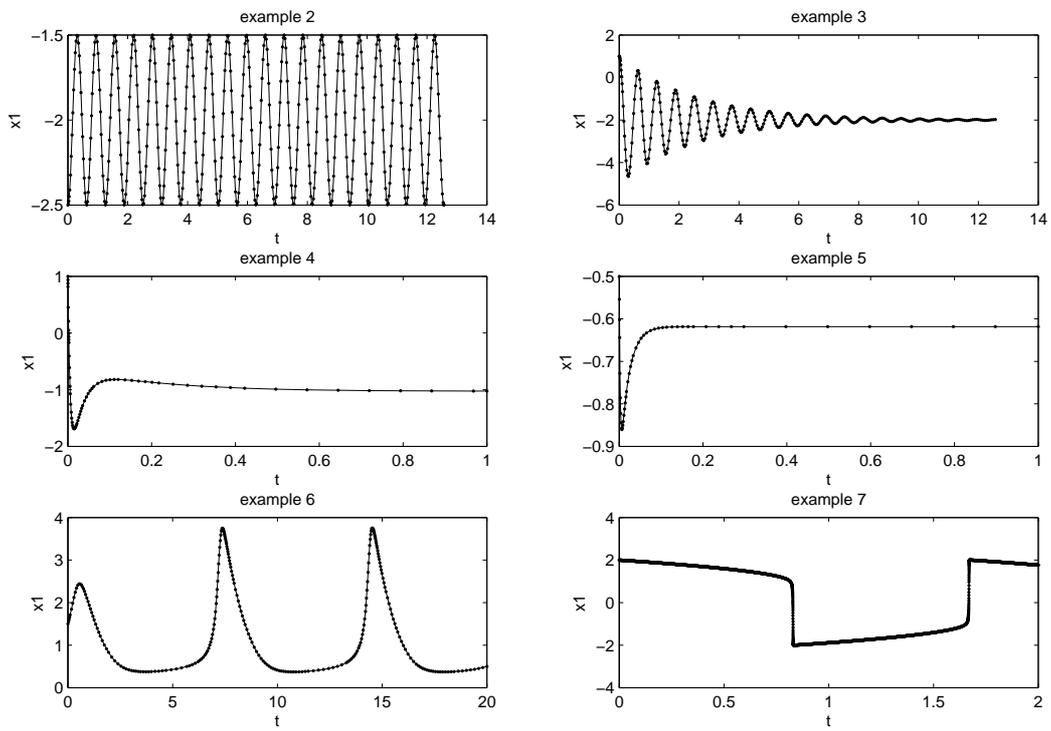}%
\caption{Path of the variables $\mathbf{x}_{1}$ (solid line)
and its approximation $\mathbf{y}_{1}$ (dots) obtained by the LLRK4 scheme in
the integration of the ODEs of examples 2-7. The time partition $(t)_{h}$ used
in each case for $\mathbf{y}_{1}$ is specified in Table II. The "exact" path
of $\mathbf{x}_{1}$ is computed with the Matlab code ode15s with
$RT=RA=10^{-13}$ on a very thin partition.}%
\end{center}
\end{figure}

The results of the integration of last two equations for each scheme
are also shown in Table II and Figure $2$. For these equations, the
relative error of the solutions obtained by the LLRK$4$ scheme is
much lower that
those of the LL$2$, but quite similar to those of the codes ode$45$ and ode$%
15s$ (which have higher order of convergence). This indicates that the LLRK$%
4 $ scheme is also appropriate for integrating non-stiff and
mild-stiff nonlinear problems as well.

In summary, results of Table II clearly indicate that the non adaptive
implementation of the LLRK$4$ scheme provides similar or much better
accuracy than the Matlab codes with equal or lower number of steps in the
integration of variety of equations. This suggests that adaptive
implementations the LLRK discretizations might archive similar accuracy than
the Matlab codes with lower or much lower number of steps, a subject that
has been already studied in \cite{Sotolongo11,Jimenez12}.

Finally, we want to point out that equations of type (\ref{Semi-Linear})
frequently arises from the discretization of nonlinear partial differential
equations. In such a case, mild or high dimensional ODEs of that form are
obtained and, as it is obvious, LLRK schemes like (\ref{LLRK4 scheme}) based
on Pad\'{e} approximations are not appropriate. Nevertheless, because the
flexibility of the high order Local Linearization approach described in
Section \ref{Section LLA}, feasible high order LL schemes can be designed
for this purpose too. For instance, by taking into account that
\begin{equation*}
\mathbf{\phi }(t_{n},\mathbf{y}_{n};\frac{h_{n}}{2})=\mathbf{\varphi }(\frac{%
h_{n}}{2}\mathbf{f}_{\mathbf{x}}(\mathbf{y}_{n}))\mathbf{f}(\mathbf{y}_{n}),
\end{equation*}%
where $\mathbf{\varphi }(z)=(e^{z}-1)/z$, the LLRK$4$ scheme (\ref{LLRK4
scheme}) can easily modified to defined an order $4$ LLRK scheme for high
dimensional ODEs. Indeed, such scheme can be defined by the same expression (%
\ref{LLRK4 scheme}), but replacing the formulas of $\widetilde{\mathbf{\phi }%
}(t_{n},\widetilde{\mathbf{y}}_{n};\frac{h_{n}}{2})$ and $\widetilde{\mathbf{%
\phi }}(t_{n},\widetilde{\mathbf{y}}_{n};h_{n})$ by
\begin{equation*}
\widetilde{\mathbf{\phi }}(t_{n},\widetilde{\mathbf{y}}_{n};\frac{h_{n}}{2})=%
\widetilde{\mathbf{\varphi }}(\frac{h_{n}}{2}\mathbf{f}_{\mathbf{x}}(%
\widetilde{\mathbf{y}}_{n}))\mathbf{f}(\widetilde{\mathbf{y}}_{n})
\end{equation*}%
and%
\begin{equation*}
\widetilde{\mathbf{\phi }}(t_{n},\widetilde{\mathbf{y}}_{n};h_{n})=\left(
\frac{h_{n}}{4}\mathbf{f}_{\mathbf{x}}(\widetilde{\mathbf{y}}_{n})\widetilde{%
\mathbf{\varphi }}(\frac{h_{n}}{2}\mathbf{f}_{\mathbf{x}}(\widetilde{\mathbf{%
y}}_{n}))+\mathbf{I}\right) \widetilde{\mathbf{\phi }}(t_{n},\widetilde{%
\mathbf{y}}_{n};\frac{h_{n}}{2}),
\end{equation*}%
respectively, where $\widetilde{\mathbf{\varphi }}$ denotes the
approximation to $\mathbf{\varphi }$ provided by the Krylov subspace method
(see, i.e., \cite{Hochbruck-etal98}). Then, a comparison with
exponential-type integrators designed for high dimensional equations of the
form (\ref{Semi-Linear}) can be carried out, but this subject is out of the
scope of this paper.

\section{Conclusions}

In summary, this paper has shown the following: 1) the LLRK approach defines
a general class of high order A-stable explicit integrators; 2) in contrast
with others A-stable explicit methods (such as Rosenbrock or the Exponential
integrators), the RK coefficients involved in the LLRK integrators are not
constrained by any stability condition and they just need to satisfy the
usual, well-known order conditions of RK schemes, which makes the LLRK
approach more flexible and simple; 3) LLRK integrators have a number of
convenient dynamical properties as the linearization preserving and the
conservation of the exact solution dynamics around hyperbolic equilibrium
points and periodic orbits; 4) unlike the majority of the previous published
works on exponential integrators, the above mentioned convergence, stability
and dynamical properties are studied not only for the discretizations but
also for the numerical schemes that implement them in practice; 5) because
of the flexibility in the numerical implementation of the LLRK methods,
specific-purpose schemes can be designed for certain classes of ODEs, e.g.,
for stiff equations, high dimensional systems of equations, etc.; 6) order $%
4 $ LLRK formula considered in this paper provides similar or much better
accuracy than the order $5$ Matlab codes with equal or lower number of steps
in the integration of variety of equations, as well as, much better
reproduction of the dynamics of the underlying equation near stationary
hyperbolic points.

Finally, it is worth to point out that theoretical properties of the LLRK
methods studied here strongly support the results of the numerical
experiments carried out by the authors in previous works \cite{de la Cruz 06}%
, \cite{de la Cruz Ph.D. Thesis}, in which the performance of other LLRK
schemes is compared with that of existing explicit and implicit schemes.

\end{document}